\documentclass[smallcondensed]{svjour3}     
\smartqed  

\usepackage{graphicx}

\usepackage{amsmath}
\usepackage{amsfonts}
\usepackage{amssymb}
\usepackage{latexsym}
\usepackage{graphicx,color}
\usepackage{multirow}

\usepackage[usenames,dvipsnames,svgnames,table]{xcolor}

\usepackage{dsfont}

\newcommand{\be}{\begin{equation}}
\newcommand{\ee}{\end{equation}}
\newcommand{\bea}{\begin{eqnarray}}
\newcommand{\eea}{\end{eqnarray}}
\newcommand{\bel}{\begin{align}}
\newcommand{\eel}{\end{align}}

\newcommand{\nno}{\nonumber}

\newcommand{\ujump}[1]{\lfloor #1\rfloor}                        

\renewcommand{\k}{\kappa}

\newcommand{\ud}{\,\mathrm{d}}
\newcommand{\ndg}[1]{| \kern -.25mm \|{#1}| \kern -.25mm \|}
\newcommand{\ncdg}[1]{| \kern -.25mm \|{#1}| \kern -.25mm \|_{\rm DG}}
\newcommand{\nsdg}[1]{| \kern -.25mm \|{#1}| \kern -.25mm \|_{\rm s}}
\newcommand{\nstdg}[1]{| \kern -.25mm \|{#1}| \kern -.25mm \|_{L_2 (J; \mathcal{D} )}}
\newcommand{\ltwo}[2]{\|{#1}\|_{{#2}}}

\newcommand{\no}{[\kern -.8mm [}
\newcommand{\nc}{]\kern -.8mm ]}

\newcommand{\norm}[2]{\| {#1} \|_{#2}}

\begin{document}

\title{On the exponent of  exponential convergence  of   {$\boldsymbol{\lowercase{p}}$-version FEM} spaces
}
\subtitle{}

\titlerunning{On the exponent of  exponential convergence  of   {$p$-version FEM} spaces}        

\author{Zhaonan Dong }

\authorrunning{Z. Dong} 

\institute{Z. Dong \at
	Department of Mathematics, University of Leicester,   Leicester, United Kingdom\\              
              \email{zd14@le.ac.uk}           
}

\date{Received: date / Accepted: date}

\maketitle

\begin{abstract}
We study the exponent of the exponential rate of convergence in terms of the   number of degrees of freedom for various non-standard   {$p$-version} finite element spaces employing  reduced cardinality basis. 
More specifically, we show that  serendipity finite element methods  and  discontinuous Galerkin finite element methods  with  total degree $\mathcal{P}_p$ basis have a faster   exponential convergence with respect to  the number of degrees of freedom than their counterparts employing  the tensor product $\mathcal{Q}_p$ basis for quadrilateral/hexahedral elements, for piecewise  analytic problems under $p$-refinement. The above results are proven by using a new $p$-optimal error bound for  the $L^2$-orthogonal projection onto the total degree $\mathcal{P}_p$ basis, and for  the  $H^1$-projection onto the serendipity finite element space  over  tensor product elements with dimension $d\geq2$. These new $p$-optimal error bounds  lead to a larger exponent of the exponential rate of convergence  with respect to the   number of degrees of freedom. Moreover, these results show that part of the basis functions in $\mathcal{Q}_p$ basis  {plays} no roles  in achieving the $hp$-optimal error bound in the Sobolev space. The sharpness of theoretical results is also  verified  by a series of numerical examples.

\end{abstract}


\section{Introduction}

Polynomial approximation on tensor product domains plays an important role in deriving the exponential rate of convergence  with respect to the number of degrees of freedom for $hp$-version finite element methods (FEMs) \cite{gui1986h,guo1986hp1,guo1986hp2,MR946376,MR954788,MR1301344,MR3351174,schwab,melenk1999hp} and $hp$-version discontinuous Galerkin finite element methods (DGFEMs)   \cite{hss,newpaper,Wihler2002,schotzau2013hp1,schotzau2013hp2,schotzau2016hp}.   In general, the proof of  the exponential rate of convergence   usually depends on the $hp$-approximation results for  some suitable projection operators onto a  local polynomial space consisting of polynomials with degree less or equal than $p$ in  each variable (known as ${\cal Q}_p$ basis) over a tensor product element (quadrilateral/hexahedral elements), for dimension $d\geq 2$. 

The key reason for using the ${\cal Q}_p$ basis over a tensor product element is because  $hp$-optimal  approximation results for the multi-dimensional  projection operators can be derived by using the stability and approximation results of the one-dimensional  projections  via tensor product arguments.  On the other hand, the $hp$-approximation results for $L^2$-orthogonal projections onto  polynomial basis with total degree less or equal than $p$ (${\cal P}_p$ basis) and $H^1$-projections onto serendipity basis (${\cal S}_p$ basis) have not been fully explored. Typically, $hp$-error bounds for  projections  onto the  ${\cal P}_p$  or   ${\cal S}_p$ basis  {have} been  derived using the fact that there exists a $q \leq p$ such that   the bases ${\cal P}_p$ or ${\cal S}_p$ contain ${\cal Q}_q$ as a subset, together with the help of the $hp$-optimal  approximation results for the projections onto the basis  ${\cal Q}_q$, see Corollary 4.52  in\cite{schwab}.

For instance, we consider the $L^2$--norm error bound of the two-dimensional $L^2$-orthogonal projection $\Pi_{{\cal Q}_p}$ onto the ${\cal Q}_p$ basis   as an example, cf. \cite{canuto1982approximation,newpaper}.  Let $\hat{\k}=(-1,1)^2$ and   $u\in H^{l}(\hat{\k})$, $l$ is an integer with $l\geq 0$. Then, the following estimate holds,
\begin{equation}\label{sec1: ex1}
\norm{u -  \Pi_{{\cal Q}_p}u }{L^2({\hat{\kappa}})}^2   \leq   C(s)( {p+1})^{-2s}
|u|_{H^{s}(\hat{\kappa})}^2,
\end{equation}
where the constant $C(s)$ is independent of $p$ and   $0\leq s  \leq \min\{p+1,l \}$. 
It is straightforward to see that   the above error bound is sharp in the sense that it is $p$-optimal in both Sobolev regularity index $l$ and polynomial approximation order $p$.

Next, we consider the $L^2$--norm error bound of $L^2$-orthogonal projection $\Pi_{{\cal P}_p}$  onto the  ${\cal P}_p$ basis. Following the Lemma 6 in \cite{MR1974174}, we define $\Pi_{{\cal P}_p} = \Pi_{{\cal Q}_{\ujump{p/2}}}$, with $\ujump{p/2}$ denoting the  largest integer which is less than or equal to $p/2$. Then, the following bound holds:
\begin{equation}\label{sec1: ex2}
\norm{u - \Pi_{{\cal P}_{p}} u }{L^2({\hat{\kappa}})}^2=\norm{u - \Pi_{{\cal Q}_{\ujump{p/2}}} u }{L^2({\hat{\kappa}})}^2   \leq   \tilde{C}(s)( {\ujump{p/2}+1})^{-2s}
|u|_{H^{s}(\hat{\kappa})}^2
,
\end{equation}
where the constant $\tilde{C}(s)$ is independent of $p$ and   $0\leq s \leq \min\{\ujump{p/2}+1,l \}$.  {We emphasize that  \emph{for function $u\in H^l(\hat{\k})$, with $p$ sufficiently large, the above error bound is $p$-optimal because $s = l$. However, if function $u$ is sufficiently smooth or even analytic, then the above error bound is $p$-suboptimal by at least ${p/2}$ orders because $s  = {\ujump{p/2}+1}$. } The similar $p$-suboptimal error bound holds for $H^1$-projections onto  ${\cal S}_p$ basis. }


Using the $p$-suboptimal error bound for $L^2$-orthogonal projections onto the ${\cal P}_p$ basis and $H^1$-projections onto  the ${\cal S}_p$ basis, it is  possible to derive an exponential rate of convergence for $hp$-FEMs employing the ${\cal S}_p$ basis  and $hp$-DGFEMs employing the ${\cal P}_p$ basis, \emph{but the resulting exponent   is much smaller with respect to the number of degrees of freedom than the exponent of  FEMs and DGFEMs employing the ${\cal Q}_p$ basis.} This contradicts the numerical observation in  work \cite{cangiani2013hp,cangiani2015hp,cangiani2016hp,Zhaonan}, where it is observed that  the error with respect to  the number of degrees of freedom for DGFEMs with the ${\cal P}_p$ basis on tensor product elements
has a steeper exponential convergence compared to  DGFEMs with the ${\cal Q}_p$ basis, for sufficiently smooth problems. This situation has been numerically tested on many different examples. We also observed numerically  that the ratio of the slope of the exponential error decay for DGFEMs with  the ${\cal P}_p$ basis compared to that of the ${\cal Q}_p$ basis depends only on the space dimension.  The same phenomenon is also observed when comparing  conforming FEMs with the ${\cal S}_p$ basis and the  ${\cal Q}_p$ basis. 

The disagreement  between the numerical observations  and theoretical results implies that the error bound \eqref{sec1: ex2} is not a sharp bound for  ${\cal P}_p$ and ${\cal S}_p$ basis. To address this, in this work, we  derive an $hp$-optimal  error bound for the $L^2$-orthogonal projection onto the ${\cal P}_p$ basis in the $L^2$--norm, and   for  the $H^1$-projection onto the ${\cal S}_p$ basis in the $L^2$--norm and $H^1$--seminorm.

The technique for proving the new error  bounds is different from the existing techniques for $hp$-approximation with the ${\cal Q}_p$ basis, due to the lack of a  tensor product structure in the ${\cal P}_p$  and  ${\cal S}_p$ bases, thereby hindering  the  use of the usual  tensor product arguments. The key tools used in this work are: a multi-dimensional orthogonal polynomial expansion and  the careful selection of basis functions.  To the best author's knowledge, the new error bounds for both projections never appeared in the literatures.  The resulting bounds are \emph{$hp$-optimal} with respect to both Sobolev regularity  and polynomial approximation order.  Moreover, it also shows that  the ${\cal Q}_p$ basis contains in a sense  ``extra'' basis functions that are
unnecessary for optimal convergence. These basis functions  do not increase the order in $p$ of the error bound, but instead only reduce its ``constant".


 By using the new $hp$-optimal error bound for the  $L^2$-orthogonal projection onto the ${\cal P}_p$ basis and the $H^1$-projection onto the ${\cal S}_p$ basis, we can prove that methods using  ${\cal P}_p$ and  ${\cal S}_p$ bases offer exponential convergence with a larger exponent with respect to  the number of degrees of freedom than comparable methods using  ${\cal Q}_p$ basis for piecewise analytic problem under $p$-refinement.   Furthermore,  the approximation results also show that there are a lot of basis functions in ${\cal Q}_p$ basis  with no roles in improving the $hp$-optimal error bound, which can be generalized to  other FEM with the local polynomial space employing reduced cardinality basis.     {Finally, we emphasize that we are using DGFEM employing ${\cal P}_p$ basis for quadrilateral and hexahedral elements  also and this is the key novelty of the approach, since this is possible for DGFEM and essentially for serendipity spaces. }

The remainder of this work is structured as follows. In Section \ref{weighted Sobolev spaces}, we introduce the required  notation and  the weighted Sobolev spaces together with some properties about the orthogonal polynomials. Then, the $p$-optimal error bound for  the $L^2$-orthogonal projection onto the ${\cal P}_p$ basis in $L^2$--norm is proved in Section 3.   In Section 4, we derive the $p$-optimal error bound for  $H^1$-projection onto the ${\cal S}_p$ basis in both  $L^2$--norm and $H^1$--seminorm. Section 5 is devoted to deriving  the exponential rate of convergence for the  $L^2$- and the $H^1$-projections employing different local polynomial bases.  The sharpness of the approximation results  is  verified through a series of numerical examples in Section 6.  

\section{Preliminaries}  \label{weighted Sobolev spaces}




\subsection{Notation}
We employ the multi-indices  ${i}=(i_1, \dots, i_d) $, and ${\alpha}=(\alpha_1,  \dots, \alpha_d)$, where each component is  non-negative.   We denote by $|\cdot|$   the $l_1$--norm of the multi-index ${i}$, with $|{i}| = \sum_{j=1}^d |i_k|$. Further, for multi-indices, the relation $ {i}\geq {\alpha}$ means that $ i_k\geq \alpha_k $ for all $ k=1,\dots,d$.  

Next, we define the following shorthand notation for the summations of indices. For   multi-indices   $i$ and $\alpha$  satisfying  $i \geq \alpha$ we define
$
\sum_{i\geq \alpha}^\infty:=
\sum_{i_1 = \alpha_1}^\infty   \dots  \sum_{i_d = \alpha_d}^\infty
$
and the summation for  multi-indices $i$ satisfying  $|i|\geq p$ is defined as 
$
\sum_{ |i| = p}^\infty.
$
Moreover, we also define a  summation for a multi-index $i$ satisfying multiple conditions, e.g. multi-index $i$ satisfying the  condition $ i \geq \alpha $ and  the condition $|i|\geq p$ is defined  as 
$
\sum_{|i| = p, i\geq \alpha}^\infty.
$

We introduce a function $\Phi_d(m,n)$ which will be used frequently in this work, given by
\begin{align} \label{Gamma-function}
\Phi_d(m,n) =\Big( \frac{\Gamma(\frac{m-n}{d}+1)}{\Gamma(\frac{m+n}{d}+1)} \Big)^d,
\end{align}
where $\Gamma$ is the Gamma function satisfying $\Gamma(n+1)=n!$ for integer  $n\geq 0$.

\subsection{Weighted Sobolev spaces}


For the reference element $\hat{\k}:=(-1,1)^d$, let 
$
W^{\alpha} ( {x}) = \prod_{k=1}^d W_k({x}_k) ^{\alpha_k},
$
where the weight function 
$
W_k(x_k):=(1-x^2_k)^{1/2},
$
 for $k=1,\dots,d$,  and   $\alpha_k \geq0$ are integers.
 
Next, we define the weighted Sobolev spaces $V^{l}(\hat{\k})$ as a closure of $C^\infty(\hat{\k})$  in the norm with the   weights $W^\alpha$, defined by
\begin{equation}
\norm{u}{V^{l}(\hat{\k})}^2 = \sum_{|\alpha|=0}^l  |u|_{V^{l}(\hat{\k})}^2, \quad \text{and}\quad  |u|_{V^{l}(\hat{\k})}^2 = \sum_{|\alpha|=l}  \norm{ W^{\alpha} D^{ \alpha} u}{L^2(\hat{\k})}^2.
\end{equation}
It is easy to see that  $|u|_{V^{l}(\hat{\k})}\leq |u|_{H^{l}(\hat{\k})}$, $\forall u\in H^{l}(\hat{\k})$, with some integer $l \geq0$. We note that the above definition for  weighted Sobolev spaces can be extended to the fractional order weighted Sobolev spaces and weighted Besov spaces  by using the real  interpolation  techniques, cf. \cite{MR0482275}.  

For $u\in L^2(\hat{\kappa})$,  we introduce the Legendre polynomial  expansion over the reference element $\hat{\k}$, given by 
$
u(x) = \sum_{|i| = 0}^\infty a_{i}\prod_{k=1}^d L_{i_k}(x_k),
$
where $x=(x_1,\dots,x_d)$, and  $L_{i_k}({x}_k)$  denotes the Legendre polynomial with order $i_k$ over the variable $x_k$. The coefficients  $a_{i}$ are defined by  
\begin{align} \label{ch7:coefficients}
a_{i} = \int_{\hat{\k}} u({x}) \prod_{k=1}^d   \frac{2i_k+1}{2} L_{i_k}({x}_k) \ud {x}.
\end{align}
The derivatives of the function $u$ can be expressed as
\begin{align} \label{expansion-derivative}
\hspace{0.5cm}
D^\alpha u(x) 
= \sum_{ i \geq \alpha }^\infty  a_{i}\prod_{k=1}^d  L_{i_k}^{(\alpha_k)}({x}_k).
\end{align}
 The derivatives of the Legendre polynomials satisfy the orthogonality property
\begin{align} \label{orthogonal-relation} 
\int_{-1}^1 (1-\xi^2)^k L^{(k)}_i(\xi)  L^{(k)}_j(\xi)  \ud \xi = \frac{2 \delta_{ij}}{2i+1} \frac{\Gamma(i+k+1)}{\Gamma(i-k+1)},
\end{align}
see  \cite[Lemma 3.10]{schwab}. By employing  \eqref{orthogonal-relation}, we have 
\begin{align} \label{weighted-seminorm}
\norm{ W^{\alpha} D^{ \alpha} u}{L^2(\hat{\k})}^2 = 
\sum_{ i \geq \alpha }^\infty
|a_{i}|^2 \prod_{k=1}^d  \frac{2}{2i_k+1} \frac{\Gamma(i_k+\alpha_k +1)}{\Gamma(i_k-\alpha_k +1)}.
\end{align}
Identity  \eqref{weighted-seminorm} establishes a link between the derivatives of the functions in the weighted $L^2$--norms and their Legendre polynomial expansions.
\begin{remark}
The weighted Sobolev space in the above definition is a special case of the general Jacobi-weighted Sobolev spaces introduced  in \cite{babuska2002direct_part1}.
The key reason to introduce the Jacobi-weighted Sobolev spaces is to  deal with the loss of orthogonality suffered by  orthogonal polynomials in standard Sobolev spaces; the $L^2$-orthogonality is  preserved in Jacobi-weighted Sobolev spaces. As we shall see in the forthcoming analysis, orthogonality plays a key role in deriving  optimal error bounds in the polynomial order $p$.
\end{remark}


\section{The $L^2$-orthogonal projection operator  onto the ${\cal P}_p$ basis} 

In this section, we derive an $hp$-optimal error bound for  the $L^2$-orthogonal projection over the reference element $\hat{\k}:=(-1,1)^d$. 

\subsection{The $L^2$-orthogonal projection operator}
For the reference element $\hat{\k}$, we define ${\cal P}_p(\hat{\k})$ and ${\cal Q}_p(\hat{\k})$ be the space of all polynomials with total degree less than or equal to $p$ and with seperate degree less than or equal to $p$, respectively. 

In order to distinguish the same projections onto spaces with different polynomial basis, we  use subscripts to signify the basis type: we use $\Pi_{{\cal Q}_p}:=\Pi_p^{(1)}\Pi_p^{(2)}\dots\Pi_p^{(d)}$ to denote the  $L^2$-projection   onto  ${\cal Q}_p$, which is constructed by using tensor product arguments together with the one dimensional  $L^2$-projection with respect to variable $x_k$, given by   $\Pi_p^{(k)}$. On the other hand, the $L^2$-projection  onto ${\cal P}_p$ is denoted by $\Pi_{{\cal P}_p}$. 

First, we have the following $hp$-optimal approximation result for the $L^2$-orthogonal projection $\Pi_{{\cal Q}_p}$, c.f. \cite[Lemma 3.4]{newpaper}.

\begin{lemma}\label{L2-Q-basis}
Let $\hat{\kappa}=(-1,1)^d$. Suppose  that $u \in H^{l}(\hat{\kappa})$,
for some interger $l\ge 0$. Let $\Pi_{{\cal Q}_p} u$ be the $L^2$-projection of $u$ onto $\mathcal{Q}_{p}(\hat{\kappa})$ with $p\geq 0$. Then, for any integer $s$, with $0 \leq s \leq \min\{p+1, l\}$, and $W_k = W_k({x}_k)$, we have:
\begin{align}\label{approxL2-Q}
\norm{u -  \Pi_{{\cal Q}_{p}}u }{L^2({\hat{\kappa}})}^2  \le   
\Phi_1(p+1,s)
\Big( \sum_{k=1}^d  \ltwo{W^s_k D^s_k u}{\hat{\kappa}}^2  \Big) \leq 
 \Phi_1(p+1,s)
 |u|_{H^{s}(\hat{\kappa})}^2 , 
\end{align}
where $ \Phi_1(p+1,s)$ is defined in \eqref{Gamma-function}.
\end{lemma} 
\begin{proof}
The result is proved  by modifying the proof of  Lemma 3.4 in \cite{newpaper}. Instead of using triangle inequality, we use   the orthogonality and stability of the one-dimensional  $L^2$-orthogonal projection, which leads to the   error bound  \eqref{approxL2-Q}.  
\end{proof}

We remark on the asymptotic behaviour of the Gamma function. Making use of  sharp double side inequalities for the Gamma function, see  Theorem 1.6. in \cite{Batir2008}, for  all positive real numbers $x\geq1$, we have  
\bea\label{Stirling's formula}
\sqrt{2\pi }  x^{x+\frac{1}{2}} e^{-x}\leq \Gamma(x+1) \leq e  x^{x+\frac{1}{2}} e^{-x},
\eea
and it follows 
\bea\label{Gamma}
\Phi_d(p+1,s) \leq C(s) \Big(\frac{d}{p+1}\Big)^{2s},
\eea
with $0\leq  s\leq \min\{p+1, l\}$ and $C(s)$ depending on the constant  $s$ only. This  implies that the error bound \eqref{approxL2-Q} is optimal in $p$ with respect to both  the Sobolev regularity index $l$ and polynomial order $p$. In fact, by modifying the proof of Theorem 6.2 in \cite{MR3556398},  it is can be  shown that the constant $C(s)=(\frac{e}{2})^{2s}$.


Next, we introduce a useful lemma which is the key tool in proving the optimal error bounds in $p$. The proof of the lemma is postponed until Section \ref{Proof of Lemma}.

\begin{lemma}\label{optimization}
Let  $\xi =  (\xi_1,\xi_2, \dots,\xi_d)$ and $\rho=(\rho_1,\rho_2, \dots,\rho_d)$ be two  non-negative integer valued vectors with  $\rho \geq \xi$,   satisfying  $|\rho| = M$, $|\xi| = m$ for $M,m\in \mathbb{N}$. Then,  we  have the (global) upper bound
\begin{align}\label{nice-bounds}
F(\xi,\rho) := \prod_{k=1}^d \frac{\Gamma(\rho_k-\xi_k+1)}{\Gamma(\rho_k+\xi_k+1)} \leq 
\Phi_d(M,m).
\end{align}
Furthermore, the  maximum value of $F(\xi,\rho)$ under the above  constraints on $\xi$ and $\rho$  is attained  at  $\xi_k = m/d$, $\rho_k = M/d$, $k=1,\dots, d$.
\end{lemma}

\begin{theorem} \label{L2-projector-pbasis}
Let $\hat{\kappa}=(-1,1)^d$. Suppose  that $u \in H^{l}(\hat{\kappa})$,
for some integer $l\ge 0$. Let $\Pi_{{\cal P}_p} u$ be the $L^2$-projection of $u$ onto $\mathcal{P}_{p}(\hat{\kappa})$ with $p\geq 0$. Then, for any integer $s$, $0 \leq s \leq \min\{p+1, l\}$, we have:
\begin{align}\label{approxL2-P}
\norm{u -  \Pi_{{\cal P}_p}u }{L^2({\hat{\kappa}})}^2  \le
\Phi_d(p+1,s)
|u|_{V^{s}(\hat{\kappa})}^2 \leq   C(s) \Big( \frac{d}{p+1}\Big)^{2s}
|u|_{H^{s}(\hat{\kappa})}^2.
\end{align}
where $ \Phi_d(p+1,s)$ is defined in \eqref{Gamma-function}.
\end{theorem}
\begin{proof}
Using the  relation   \eqref{orthogonal-relation} ,  for any integer $s$, $0 \leq s \leq \min\{p+1, l\}$, we have 
\begin{align}
  \norm{u-\Pi_{{\cal P}_p} u}{L^2(\hat{\k})}^2 &= \sum_{|i| = p+1}^\infty |a_i|^2 \prod_{k=1}^d  \frac{2}{2i_k+1} 
\leq  \sum_{|\alpha|=s} \hspace{0.2cm}   \sum_{  |i| = p+1, i \geq \alpha }^\infty |a_i|^2 \prod_{k=1}^d  \frac{2}{2i_k+1}   \nno \\
& 
 =    \sum_{|\alpha|=s} \hspace{0.2cm}    \sum_{|i| = p+1, i \geq \alpha }^\infty |a_i|^2 \Big( \prod_{k=1}^d  \frac{2}{2i_k+1}  \frac{\Gamma(i_k+\alpha_k +1)}{\Gamma(i_k-\alpha_k +1)} \Big)      \nno \\
 &
 \times
  \Big( \prod_{k=1}^d  \frac{\Gamma(i_k-\alpha_k +1)}{\Gamma(i_k+\alpha_k +1)}  \Big)  \nno \\
&
 \leq 
\Phi_d(p+1,s)
\sum_{|\alpha|=s}   \hspace{0.2cm}  \sum_{|i| = p+1, i \geq \alpha }^\infty |a_i|^2  \prod_{k=1}^d  \frac{2}{2i_k+1}\frac{\Gamma(i_k+\alpha_k +1)}{\Gamma(i_k-\alpha_k+1)}  \nno \\
&
 \leq 
\Phi_d(p+1,s)
\sum_{ |\alpha|=s}  \norm{ W^{\alpha} D^{ \alpha} u}{L^2(\hat{\k})}^2   \nno \\
&
=  
\Phi_d(p+1,s)
|u|_{V^{s}(\hat{\kappa})}^2    \leq  C(s) \Big( \frac{d}{p+1}\Big)^{2s}
|u|_{H^{s}(\hat{\kappa})}^2, \nno 
\end{align}
where in step one, the index set is enlarged; indeed,  some of the terms with multi-index $|i| \geq p+1$ have been used more than once; in step three,  we use Lemma \ref{optimization}, taking  $\xi_k =  \alpha_k \geq 0$, $\rho_k = i_k \geq 0$, $M = p+1$, $m = s$, together with the restriction $0 \leq s \leq \min\{p+1, l\}$; in step four,  we  used  \eqref{weighted-seminorm} and  in the last step the  bound holds from \eqref{Gamma}. 
\end{proof}
\begin{remark}
We point out that the above proof  for the $L^2$-orthogonal projection $\Pi_{{\cal P}_p}$ on $d$-dimensional reference element is a natural extension of the proof for one-dimensional result, see \cite{schwab} for details.

By comparing  the $L^2$--norm  bound \eqref{approxL2-Q} for the projection $\Pi_{{\cal Q}_p}$ and \eqref{approxL2-P} for the projection $\Pi_{{\cal P}_p}$, it is easy to see that both bounds are $p$-optimal with respect to  Sobolev regularity index $l$ and also for polynomial order $p$. Moreover,  we can see that  the bound in \eqref{approxL2-P} will have a larger constant compared to the bound in \eqref{approxL2-Q}, and this constant  depends on the dimension $d$. This result will play a key role in deriving the exponential convergence for the ${\cal P}_p$ basis.  
\end{remark}


\subsection{The Proof of Lemma \ref{optimization}}\label{Proof of Lemma}

The proof will be split into three steps.

{\bf Step 1:}  The proof follows a constrained optimization procedure. We set, 
\bea
L(\xi, \rho, \mu,\lambda) =F(\xi,\rho)  + \mu (|\xi| - m) +\lambda (|\rho| - M),
\eea
and we calculate the stationary points. We consider the partial derivatives with respect to  $\xi_k$ and $\rho_k$, $k=1,\dots,d$,
$$
\frac{\partial L}{\partial \xi_k} =  -\left( \frac{\Gamma^\prime(\rho_k- \xi_k+1)}{\Gamma(\rho_k - \xi_k+1)}+ \frac{\Gamma^\prime(\rho_k+ \xi_k+1)}{\Gamma(\rho_k + \xi_k+1)} \right) F(\xi,\rho)  + \mu =0,
$$
and 
$$
\frac{\partial L}{\partial \rho_k} = \left( \frac{\Gamma^\prime(\rho_k- \xi_k+1)}{\Gamma(\rho_k - \xi_k+1)}- \frac{\Gamma^\prime(\rho_k+ \xi_k+1)}{\Gamma(\rho_k + \xi_k+1)} \right) F(\xi,\rho)  + \lambda =0,
$$
which satisfy the equations
\bea\label{xi-relation}
\frac{\Gamma^\prime(\rho_k-\xi_k+1)}{\Gamma(\rho_k- \xi_k+1)}   =\frac{\mu -\lambda}{2F(\xi,\rho)} \quad  \text{and} \quad \frac{\Gamma^\prime(\rho_k+\xi_k+1)}{\Gamma(\rho_k+ \xi_k+1)}   =\frac{\mu +\lambda}{2F(\xi,\rho)} , 
\eea
with $k=1,\dots,d$,  by using the fact that $F(\xi,\rho) > 0$.  The right-hand sides of the two equations in  \eqref{xi-relation} are independent of the index $k$. Moreover, the function  $\phi(z) = \Gamma(z)^\prime/\Gamma(z)$ is the so-called Digamma function with the property that (see \cite{MR0167642}, (6.3.16))
$$
\phi(z+1) = -\gamma  + \sum_{n=1}^{\infty} \frac{z}{n(n+z)} = -\gamma  + \sum_{n=1}^{\infty} \Big( \frac{1}{n}-\frac{1}{n+z} \Big), \quad z\neq -1, -2,\dots,
$$
where  $\gamma$ is the  Euler-Mascheroni constant. For $z \geq 0$, the function $\phi(z+1)$ is a continuous monotonically increasing function, which shows that  \eqref{xi-relation} have only one solution. This solution is $\tilde{\xi}_k= m/d$ and $\tilde{\rho}_k= M/d$, $k=1,\dots,d$, and the $F(\xi,\rho)$ will have the extreme value at this stationary point, given by 
\bea\label{maxi-value}
F(\tilde{\xi},\tilde{\rho})  = 
\Phi_d(M,m).
\eea
{\bf Step 2:}  In order to find the global maximum,  we need to prove the following asymptotic relationship:
 \bea\label{maxi-value-asymptotic}
\Phi_n(M,m)
\leq 
\Phi_d(M,m), \qquad  n=1, \dots, d-1.
\eea
This is proven by considering three different cases. We first consider the  special case $m = 0$. In this case,  \eqref{maxi-value-asymptotic} holds trivially. 
  Next, we  consider the  case  $m = \delta M$, with $0< \delta < 1$. By using the property  \eqref{Stirling's formula} of Gamma functions, we have the following bound: 
\begin{align}\label{relation-useful}
\frac{
 \Phi_d(M,m)
 }
 {
 \Phi_n(M,m)
 } 
\geq 
\Big(\frac{\sqrt{2 \pi}}{e}\Big)^{d+n} \Big(\frac{d}{n}\Big)^{2\delta M} \Big( \frac{1-\delta}{1+\delta}\Big)^{\frac{d-n}{2}}.
\end{align}
By recalling that $0< \delta < 1$ and $n=1,\dots, d-1$, we have that $0<\frac{1-\delta}{1+\delta}<1$  and the function $(\frac{d}{n})^{2\delta M}$ is  monotonically increasing  with respect to $M$. For $M\geq \left( (d+n) \log({\frac{e}{\sqrt{2\pi}}}) + \frac{d-n}{2}\log({\frac{1+\delta}{1-\delta}}) \right)\left( 2\delta \log(\frac{d}{n}) \right)^{-1}$,  the above quotient formula is greater than $1$ and therefore \eqref{maxi-value-asymptotic} holds.

Finally, we consider the case $m =M$. Using the same techniques used to derive \eqref{relation-useful} together with the fact that $\Gamma(1)=1$, we have
\begin{align}\label{relation-useful2}
\frac{ 
\Phi_d(M,m)
}
{
\Phi_n(M,m)} 
= \frac{(\Gamma(\frac{2M}{n}+1))^n}{(\Gamma(\frac{2M}{d}+1))^d} \geq \frac{(\sqrt{2\pi})^n}{e^d}\Big( \frac{d}{2M}\Big)^{\frac{d-n}{2}} \Big( \frac{d}{n}\Big)^{2M+\frac{n}{2}}.
\end{align}
By using the fact that exponentially increasing functions  grow faster than polynomials, we know that for sufficiently large $M$ the right hand side of  \eqref{relation-useful2} is greater than $1$ and therefore \eqref{maxi-value-asymptotic} holds. 

{\bf Step 3:} Finally, we need to show that the extreme value \eqref{maxi-value} is the global maximum value of   $F({\xi},{\rho}) $ under the constraints $|\xi| = m$ and $|\rho| = M$.

First, we can see that the function $F({\xi},{\rho}) $ is  symmetric and  continuous with respect to  $\xi$ and $\rho$. The   constraints $|\xi| = m$ and $|\rho| = M$ restrict  the domain of  $\xi$ and $\rho$ to be a $(d-1)$-dimensional simplex, which is convex and  compact. So \emph{the maximum value of the function $F({\xi},{\rho}) $ over the domain will be obtained only at the boundary of the domain or the stationary point of $F(\xi,\rho)$.} We have calculated  the function value at the stationary point in  \eqref{maxi-value} already, so now we just need to check the function values on the boundary of the domain.

This may be proved by induction. We start with the case $d=2$: the domain of $\xi$ and $\rho$ satisfying the constrains are two straight lines, ${\rho_1}+{\rho_2}=M$ and ${\xi_1}+{\xi_2}=m$. Here, the stationary point is the mid-point of each of the two lines $\tilde{\xi}=(m/2,m/2)$, $\tilde{\rho}=(M/2,M/2)$,  and the boundary of the domain consist of the  points $\xi^{b} = (0,m)$, $\rho^{b} = (0,M)$ or $\xi^{b} = (m,0)$, $\rho^{b} = (M,0)$, due to the constraints $\rho\geq \xi$. Using the symmetry of the function and of the domain, we know that at the two boundary points of the domain, $F(\xi,\rho)$ will attain the same value, with $F(\xi^b,\rho^b)  = \Phi_1(M,m)$.  
  By using the asymptotic  relation \eqref{maxi-value-asymptotic}, we find 
$$
F(\xi^b,\rho^b)  = 
\Phi_1(M,m)
\leq 
\Phi_2(M,m)
 =F(\tilde{\xi},\tilde{\rho}).
$$
The above relation shows that the extreme value \eqref{maxi-value} is the global maximum value under the constraints for $d=2$. 

Next, we consider the case $d=3$, where the domain of each of $\xi$ and $\rho$ will be a triangle. In this case, the stationary point of $F(\xi,\rho)$ is when $\xi$ and $\rho$ are located at the barycenter of their respective triangle. The boundary of each domain consists of  $3$ straight lines. We need  to calculate the maximum value of $F(\xi,\rho)$ on the  boundary of the domain. By using the symmetry of $F(\xi,\rho)$, and that fact that $|\xi| = m$ and $|\rho|=M$, we only need to consider one part of domain boundary where $\xi_3$ = 0 and $\rho_3 = 0$. Then, the maximum of $F(\xi,\rho)$ on the domain boundary can be viewed as exactly the same problem with the same constraints as in the case $d=2$. Consequently, the maximum value of $F(\xi,\rho)$ along the boundary of the domain is $F(\xi^b,\rho^b)  = \Phi_2(M,m)$. 
Again, by using the same techniques as for $d=2$, we deduce that
$$
F(\xi^b,\rho^b)  = 
\Phi_2(M,m)
\leq 
\Phi_3(M,m)
=F(\tilde{\xi},\tilde{\rho}).
$$
The above relation shows that the extreme value \eqref{maxi-value} is the global maximum value under the constraints for $d=3$.  For the general $d$-dimensional case, the proof can be carried out  in a similar way. The key observation is that the maximum value of  $F(\xi,\rho)$ on the boundary of $d$-dimensional domain  will be at the stationary point of $F({\xi},{\rho})$ on the $(d-1)$-dimensional domain. By using the relation 
$$
\Phi_{d-1}(M,m)
\leq 
\Phi_d(M,m),
$$
the proof is complete.


\section{The $H^1$-projection operator onto  the ${\cal S}_p$ basis} 

In this section, we shall  consider the $H^1$-projection over the reference element $\hat{\k}:=(-1,1)^d$ with  $d=2,3$. Since the  three dimensional  results depend on the two dimensional results, we start with the two dimensional case.

\subsection{The $H^1$-projection operator on the reference square}

First, we   introduce the two-dimensional  serendipity finite element space, cf.  \cite{schwab}
\bea \label{serendipity space 2D}
{\cal S}_p(\hat{\k}): = {\cal P}_p(\hat{\k}) + \text{span} \{x_1^px_2, x_1x_2^p \}, \quad p\geq1.
\eea
 We can see in Figure \ref{Serendipity basis 2D}  that the serendipity space ${\cal S}_p$ contains two more basis functions than the  ${\cal P}_p$ basis for $p\geq 2$. Another way to define  the serendipity  basis is to consider the decomposition of the $C^0$ finite element space with ${\cal Q}_p$ basis over $\hat{\k}$.  For polynomial order $p$, the ${\cal S}_{p}$ basis  has  the same number of nodal basis functions and edge basis functions  as the ${\cal Q}_{p}$ basis, but the ${\cal S}_{p}$  basis only has internal moment basis functions  (those with  zero value along the element boundary $\partial \hat{\k}$) whose total degree is less than or equal  $p$, cf. \cite{schwab,MR1164869}.  We note that  serendipity FEMs can be defined in a dimension-independent fashion, see  \cite{arnold2011serendipity}. 

\begin{figure}
\begin{center}
\begin{tabular}{cc}
\hspace{-0.7 cm} \includegraphics[scale=0.25]{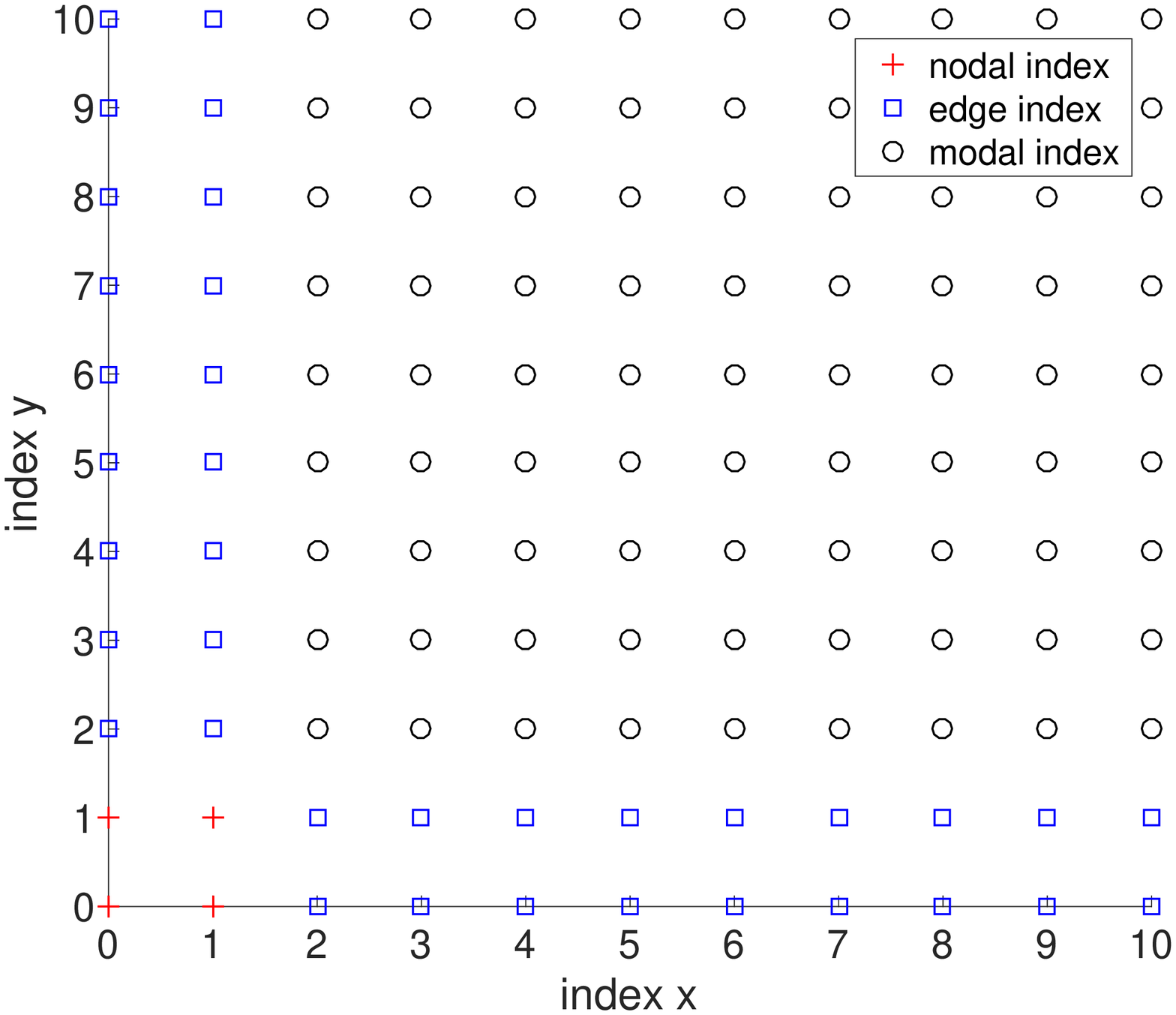} &
\hspace{-0 cm} \includegraphics[scale=0.25]{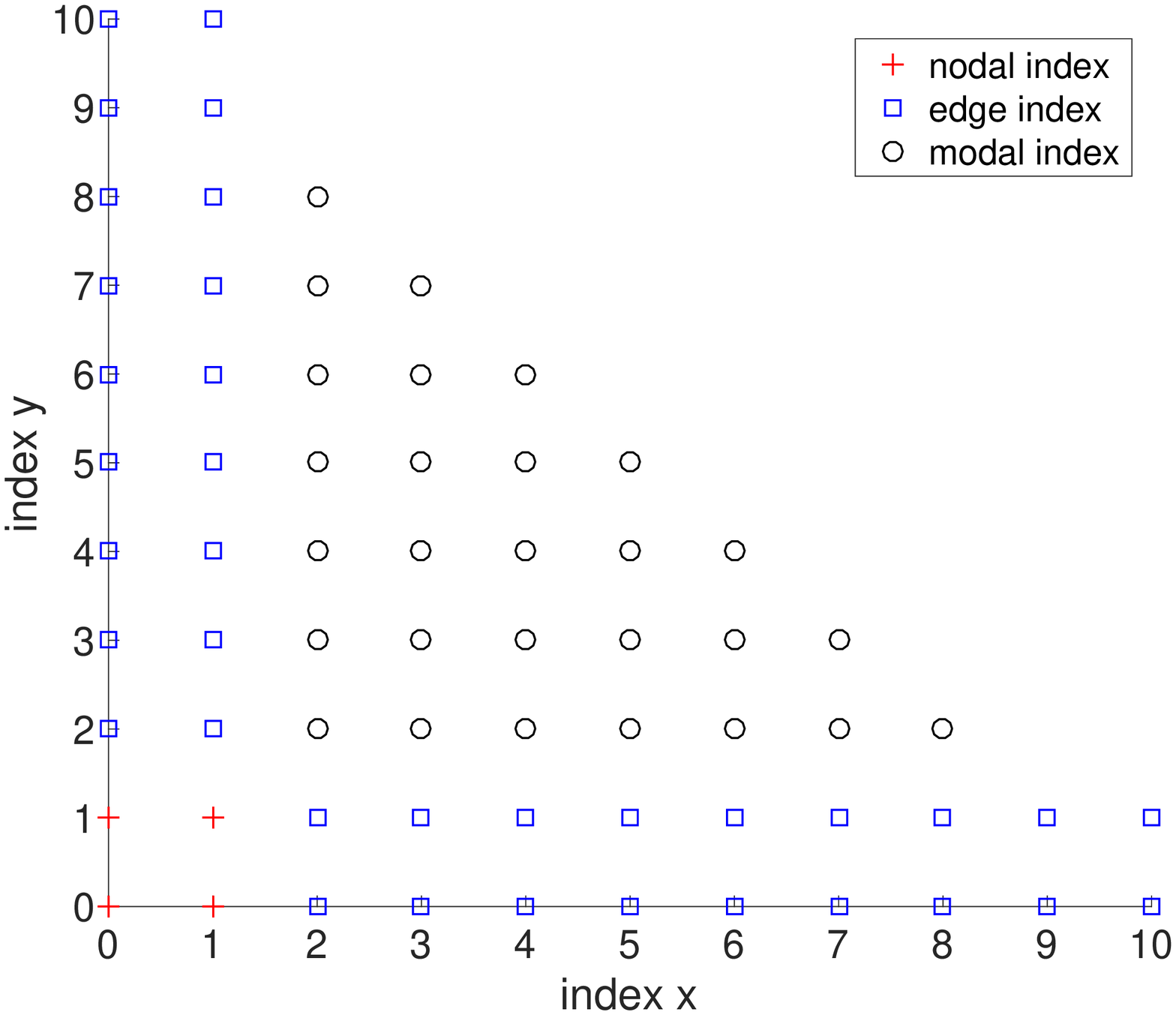}\\
\end{tabular}
\end{center}
\caption{${\cal Q}_p$ (left) and ${\cal S}_p$ (right) with polynomial order $10$.} \label{Serendipity basis 2D}
\end{figure}

Similarly to the case of the $L^2$-projection, we use $\pi_{{\cal Q}_p}:=\pi_p^{(1)}\pi_p^{(2)}$ to denote the  $H^1$-projection  onto the ${\cal Q}_p$ basis, which can be constructed via a tensor product of one dimensional  $H^1$-projection  with respect to variable $x_k$, given by   $\pi_p^{(k)}$. Similarly, the $H^1$-projection onto the ${\cal S}_p$ basis is denoted  by   $\pi_{{\cal S}_p}$, which is defined in \eqref{def:H^1 projector S basis}. 

Now, we construct the  two-dimensional  $H^1$-projection explicitly by using the    one-dimensional $H^1$-projection and tensor product arguments, see \cite{schwab,hss}.
For $u \in    H^l (\hat{\k})$, $l\geq 2$,  the projection $\pi_{{\cal Q}_p} u \in {\cal Q}_p(\hat{\k})$, $p\geq 1$, is defined by
\begin{align} \label{def:U in 2D}
 \pi_{{\cal Q}_p} u(x_1,x_2)  := & \int_{-1}^{x_1}\int_{-1}^{x_2} 
  \Pi_{{\cal Q}_{p-1}} \partial_1  \partial_2 u(x_1,x_2) \ud x_1  \ud x_2  +\int_{-1}^{x_1}    \Pi_{p-1}^{(1)} \partial_1  u(x_1,-1)  \ud x_1  
 \nno \\
&  + \int_{-1}^{x_2}    \Pi_{p-1}^{(2)} \partial_2 u(-1,x_2)   \ud x_2  +u(-1,-1)  \nno \\
=& \sum_{i_1=0}^{p-1}\sum_{i_2=0}^{p-1} a_{i_1 i_2}\psi_{i_1} (x_1)\psi_{i_2}(x_2)  
+ \sum_{{i_1}=0}^{p-1} b_{{i_1}}\psi_{i_1}(x_1)   
+\sum_{{i_2}=0}^{p-1} c_{{i_2}}\psi_{i_2}(x_2)  \nno \\
& +u(-1,-1);
\end{align}
the  projection $ \Pi_{{\cal Q}_{p-1}}$ and $\Pi_{p-1}^{(k)}$ are the two dimensional and one dimensional $L^2$-orthogonal projections, respectively, the coefficients  $a_{{i_1}{i_2}}$, $b_{i_1}$ and  $c_{i_2}$ are given by: 
\begin{align} \label{coefficients 2D}
a_{{i_1}{i_2}} &= \frac{2{i_1}+1}{2}\frac{2{i_2}+1}{2} \int_{\hat{\k}}    \partial_1  \partial_2 u(x_1,x_2) L_{i_1}(x_1) L_{i_2}(x_2)\ud x,  \nno \\
b_{i_1} &= \frac{2{i_1}+1}{2} \int_{-1}^1    \partial_1  u(x_1,-1) L_{i_1}(x_1)\ud x_1, \nno \\
c_{i_2} &= \frac{2{i_2}+1}{2} \int_{-1}^1    \partial_2  u(-1,x_2) L_{i_2}(x_2)\ud x_2 ,
\end{align}
 and the polynomial  function  $\psi_{j}(z)=\int_{-1}^z L_{j}(z)  \ud z$  with degree $j+1$, and satisfies $\psi_{j}(\pm1) = 0$ for $j\geq 1$.  Moreover,  for $j\geq 1$, $\psi_j(z) = -\frac{1}{j(j+1)}(1-z^2)L_j^\prime(z)$ has the following properties, cf. \cite{schwab},
\begin{align}\label{H1 orthogonal-1}
\int_{-1}^{1} \psi_j(z)\psi_k(z)\frac{1}{1-z^2} \ud z = \frac{2\delta_{jk}}{j(j+1)(2i+1)}.
\end{align}

Next, we rearrange the relation \eqref{def:U in 2D} by  separating  the internal moment  basis functions:  
\begin{align} \label{def:H^1 projector Q basis}
\pi_{{\cal Q}_p} u (x_1,x_2)
: =& \sum_{i_1=1}^{p-1}\sum_{i_2=1}^{p-1} a_{{i_1}{i_2}}
  \psi_{i_1}(x_1)\psi_{i_2}(x_2)  
+\sum_{{i_1}=0}^{p-1}  a_{{i_1}0}\psi_{i_1}(x_1)\psi_0(x_2) +u(-1,-1) 
 \nno \\
&+ \sum_{{i_2}=1}^{p-1}  a_{0{i_2}}\psi_0(x_1)\psi_{i_2}(x_2)
+\sum_{{i_1}=0}^{p-1} b_{{i_1}}\psi_{i_1}(x_1) +\sum_{{i_2}=0}^{p-1} c_{{i_2}}\psi_{i_2}(x_2) , 
\end{align}
so that the first  double summation  in \eqref{def:H^1 projector Q basis} only contains the internal moment basis functions.  From the definition of ${\cal S}_p$, $\pi_{{\cal S}_p}$ can be constructed by removing the internal moment  basis functions with polynomial order greater than $p$  in $\pi_{{\cal Q}_p}$. More specifically, $\pi_{{\cal S}_p} u \in {\cal S}_p(\hat{\k})$, $p\geq 4$, is defined by
\begin{align} \label{def:H^1 projector S basis}
\pi_{{\cal S}_p}u(x_1,x_2) :=&  \sum_{\substack {  |{i}| = 2 \\{i_k}\geq1, k=1,2 }}^{p-2 } a_{{i_1}{i_2}}\psi_{i_1}(x_1)\psi_{i_2}(x_2)  
+\sum_{{i_1}=0}^{p-1}  a_{{i_1}0}\psi_{i_1}(x_1)\psi_0(x_2)
+u(-1,-1)
  \nno \\
& 
+ \sum_{{i_2}=1}^{p-1}  a_{0{i_2}}\psi_0(x_1)\psi_{i_2}(x_2) 
+ \sum_{{i_1}=0}^{p-1} b_{{i_1}}\psi_{i_1}(x_1) +\sum_{{i_2}=0}^{p-1} c_{{i_2}}\psi_{i_2}(x_2). 
\end{align}
For $1\leq p\leq 3$, the first term in \eqref{def:H^1 projector S basis} will vanish, because there are no internal moment  basis functions for the serendipity basis in that case. In this work,  we focus on the high order polynomial cases, so we only consider the  $H^1$-projection $\pi_{{\cal S}_p}$ for $p\geq 4$.

Next,  we recall the following approximation lemma for $\pi_{{\cal Q}_p}$ from  \cite{hss}. 
\begin{lemma}\label{H^1 projector Q basis}
Let $\hat{\kappa}=(-1,1)^2$. Suppose that $u \in H^{l+1}(\hat{\kappa})$,
for some $l\ge 1$. Let $\pi_{{\cal Q}_p} u$ be the $H^1$-projection of $u$ onto $\mathcal{Q}_{p}(\hat{\kappa})$ with $p\geq 1$. Then, we have
\bea
\pi_{{\cal Q}_p} u = u \quad \text{at the vertices of $\hat{\k}$},
\eea
and the following error estimates hold:
\begin{align}\label{H1-projector-L2 norm}
\norm{u - \pi_{{\cal Q}_p} u}{L^2(\hat{\k})}^2& \leq
\frac{2}{p(p+1)}
\Phi_1(p,s)
 \Big( \norm{\partial_1^{s+1} u}{L^2(\hat{\k})}^2
+2\norm{\partial_2^{s+1} u}{L^2(\hat{\k})}^2 \Big) \nno \\
& + \frac{4}{p^2(p+1)^2}
 \Phi_1(p,s-1)
\norm{\partial_1 \partial_2^{s} u}{L^2(\hat{\k})}^2  
\leq 
C(s)\Big(  \frac{1}{p}\Big)^{2s+2}|u |^2_{H^{s+1}(\hat{\k})}, 
\end{align}
and
\begin{align}\label{H1-projector-H1 norm}
&\norm{\nabla (u - \pi_{{\cal Q}_p} u)}{L^2(\hat{\k})}^2 
 \leq 
2
 \Phi_1(p,s)
\Big( \norm{\partial_1^{s+1} u}{L^2(\hat{\k})}^2
+\norm{\partial_2^{s+1} u}{L^2(\hat{\k})}^2 \Big) \nno \\
& \hspace{0cm} + \frac{8}{p(p+1)} 
 \Phi_1(p,s-1)
\Big( \norm{\partial_1^{s}\partial_2 u}{L^2(\hat{\k})}^2
+\norm{\partial_1 \partial_2^{s} u}{L^2(\hat{\k})}^2 \Big)  
 \leq 
C(s)\Big(  \frac{1}{p}\Big)^{2s}|u |^2_{H^{s+1}(\hat{\k})},
\end{align}
for any integer $s$, $1 \leq s \leq \min\{p, l\}$.
\end{lemma}

Then, we derive the $hp$-error bound for the $H^1$-projection $\pi_{{\cal S}_p}$ for $p\geq 4$. 

\begin{theorem}\label{H^1 projector S basis}
Let $\hat{\kappa}=(-1,1)^2$. Suppose  that $u \in H^{l+1}(\hat{\kappa})$,
for some $l\ge 1$. Let $\pi_{{\cal S}_p} u$ be the $H^1$ projection of $u$ onto $\mathcal{S}_{p}(\hat{\kappa})$ with $p\geq 4$. Then, we have
\bea \label{relation 1 S point value}
\pi_{{\cal S}_p} u = u \quad \text{at the vertices of $\hat{\k}$},
\eea
and for any integer $s$, $1 \leq s \leq \min\{p, l\}$,  {$p$ sufficiently large}, the following error estimates hold:
\begin{align}\label{H1-projector-L2 norm S}
 \norm{u - \pi_{{\cal S}_p} u}{L^2(\hat{\k})}^2 &\leq
\frac{4}{p(p+1)}
 \Phi_1(p,s)
 \Big( \norm{\partial_1^{s+1} u}{L^2(\hat{\k})}^2
+2\norm{\partial_2^{s+1} u}{L^2(\hat{\k})}^2 \Big)\nno   \\
&  + \frac{8}{p^2(p+1)^2} 
 \Phi_1(p,s-1)
\norm{\partial_1 \partial_2^{s} u}{L^2(\hat{\k})}^2   \nno   \\
& \hspace{0cm} +72 
 \Phi_2(p+1,s+1)
 |\partial_1\partial_2 u |^2_{V^{s-1}(\hat{\k})}    
  \leq C(s)\Big(  \frac{2}{p+1}\Big)^{2s+2}|u |^2_{H^{s+1}(\hat{\k})},
\end{align}
and
\begin{align}\label{H1-projector-H1 norm S}
 \norm{\nabla (u - \pi_{{\cal S}_p} u)}{L^2(\hat{\k})}^2 
&\leq 
4
 \Phi_1(p,s)
\Big( \norm{\partial_1^{s+1} u}{L^2(\hat{\k})}^2
+\norm{\partial_2^{s+1} u}{L^2(\hat{\k})}^2 \Big)  \nno \\
& + \frac{16}{p(p+1)} 
 \Phi_1(p,s-1)
 \Big( \norm{\partial_1^{s}\partial_2 u}{L^2(\hat{\k})}^2
 +
 \norm{\partial_1 \partial_2^{s} u}{L^2(\hat{\k})}^2 \Big)
  \nno \\
& +24 
 \Phi_2(p,s)
|\partial_1\partial_2 u |^2_{V^{s-1}(\hat{\k})} 
\leq C(s)\Big(  \frac{2}{p}\Big)^{2s}|u |^2_{H^{s+1}(\hat{\k})}.
\end{align}
\end{theorem}
\begin{proof}
The key observation is  the fact that the serendipity basis $\mathcal{S}_{p}$ differs from the   $\mathcal{Q}_{p}$ basis only at the internal moment  basis functions which vanish along the boundary of $\hat{\k}$.  Indeed, using  \eqref{def:H^1 projector Q basis} and  \eqref{def:H^1 projector S basis}, we have 
\begin{align}
\Big(\pi_{{\cal Q}_p} u - \pi_{{\cal S}_p} u \Big)(x_1,x_2) = 
\sum_{\substack {  |i| =  p-1 \\p-1 \geq {i_k} \geq1,k=1,2 }}^{2(p-1)} 
a_{{i_1}{i_2}}\psi_{i_1}(x_1)\psi_{i_2}(x_2).
\end{align}
Using the fact that   $\psi_j(\pm 1)=0$ for $j\geq1$, we deduce that $(\pi_{{\cal Q}_p} u - \pi_{{\cal S}_p} u)|_{\partial{\hat{\k}}}=0$.  Thus, \eqref{relation 1 S point value} is proved.

Next, we derive \eqref{H1-projector-L2 norm S}. The first step is the use of the triangle inequality, 
\begin{align}
 \norm{u - \pi_{{\cal S}_p} u}{L^2(\hat{\k})}^2 \leq 2 \norm{u - \pi_{{\cal Q}_p} u}{L^2(\hat{\k})}^2 + 2\norm{\pi_{{\cal Q}_p} u -\pi_{{\cal S}_p} u}{L^2(\hat{\k})}^2. 
\end{align}
Thus, we only need to consider the error from the  second term in  the above bound. 
 By using  the orthogonality relation \eqref{H1 orthogonal-1} of $\psi_j(x)$ for $j\geq 1$ and $1 \leq s \leq \min\{p, l\}$, we have 
\begin{align} \label{relation1-H1}
\norm{ \pi_{{\cal Q}_p} u - \pi_{{\cal S}_p} u}{L^2(\hat{\k})}^2 &\leq \norm{ (\pi_{{\cal Q}_p}  u - \pi_{{\cal S}_p} u)W^{-1} }{L^2(\hat{\k})}^2 \nno \\
&= 
\sum_{\substack {  |i| =  p-1 \\p-1 \geq {i_k} \geq1,k=1,2 }}^{2(p-1)} 
|a_{{i_1}{i_2}}|^2 \prod_{k=1}^2 \frac{2}{2{i_k}+1} \frac{1}{{i_k}({i_k}+1)} \nno  \\
&\leq \sum_{|\alpha| = s-1} \hspace{0.2cm} \sum_{\substack { |i| = p-1, {i} \geq \alpha  \\ 
i_k\geq 1, k=1,2 }}^\infty 
|a_{{i_1}{i_2}}|^2 
\prod_{k=1}^2 \frac{2}{2{i_k}+1} \frac{1}{{i_k}({i_k}+1)} 
,
\end{align}
where in step two, we enlarged the summation index sets by adding the high order internal moment  basis functions with coefficients $a_{i_1i_2}$, $i_k\geq 1$ for $k=1,2$ and $|i|\geq p-1$. Thus, we have
\begin{align} \label{2D Q-S}
\norm{ \pi_{{\cal Q}_p} u - \pi_{{\cal S}_p} u}{L^2(\hat{\k})}^2 
&\leq 
\sum_{|\alpha| = s-1}  \hspace{0.2cm} \sum_{\substack { |i| = p-1, {i} \geq \alpha  \\ 
i_k\geq 1, k=1,2 }}^\infty 
 |a_{i_1 i_2}|^2
 \Big( \prod_{k=1}^2 \frac{2}{2{i_k}+1}  \frac{\Gamma(i_k + \alpha_k+1)}{\Gamma(i_k - \alpha_k+1)}    \Big)  \nno \\ 
& \times \Big( \prod_{k=1}^2   \frac{1}{{i_k}({i_k}+1)} \frac{\Gamma(i_k-\alpha_k+1) }{\Gamma(i_k+\alpha_k+1)}   \Big)  \nno \\
& \leq 
  \sum_{|\alpha| = s-1} \hspace{0.2cm} \sum_{\substack { |i| = p-1, {i} \geq \alpha  \\ 
i_k\geq 1, k=1,2 }}^\infty 
 |a_{i_1 i_2}|^2
 \Big( \prod_{k=1}^2 \frac{2}{2{i_k}+1}  \frac{\Gamma(i_k + \alpha_k+1)}{\Gamma(i_k - \alpha_k+1)}    \Big)   \nno \\
& \times \Big( \prod_{k=1}^2   \frac{\Gamma(i_k-\alpha_k+1) }{\Gamma(i_k+\alpha_k+3)}   \Big)\times36 .
\end{align}
Where we used $\frac{1}{i_k (i_k+1)} \leq \frac{6}{(i_k+\alpha_k+1)(i_k+\alpha_k+2)}$, since $i_k \geq \alpha_k$ and $i_k\geq1$. 
Now,  we have 
\begin{align} \label{2D Q-S step 2}
\norm{ \pi_{{\cal Q}_p} u - \pi_{{\cal S}_p} u}{L^2(\hat{\k})}^2 
& \leq \sum_{|\alpha| = s-1} \hspace{0cm}
  \sum_{\substack { |i| = p-1, {i} \geq \alpha   }}^\infty 
 |a_{i_1 i_2}|^2
 \Big( \prod_{k=1}^2 \frac{2}{2{i_k}+1}  \frac{\Gamma(i_k + \alpha_k+1)}{\Gamma(i_k - \alpha_k+1)}    \Big)  \nno \\
& \times \Big( \prod_{k=1}^2   \frac{\Gamma(i_k-\alpha_k+1) }{\Gamma(i_k+\alpha_k+3)}   \Big) \times 36   \nno \\
&\leq 36
 \Phi_2(p+1,s+1)
 \sum_{|\alpha| = s-1}  \norm{ W^{\alpha} D^{ \alpha} (\partial_1\partial_2u)}{L^2(\hat{\k})}^2  \nno \\
& = 36 
 \Phi_2(p+1,s+1)
|\partial_1\partial_2 u |^2_{V^{s-1}(\hat{\k})} 
\leq C(s)\Big(  \frac{2}{p+1}\Big)^{2s+2}\hspace{-0.2cm} |u |^2_{H^{s+1}(\hat{\k})};
\end{align}
 in step one, we enlarge the index set by adding functions with coefficients $a_{i_1i_2}$ whose index satisfying the relation $|i|\geq p-1$, $\prod_{k=1}^2 i_k =0$,  
 while in step two, we use  Lemma \ref{optimization}, with $\xi_1 =  \alpha_1 +1 \geq 1$, $\xi_2 = \alpha_2+1 \geq 1$, $\rho_1 = i_1 + 1 \geq 1$, $\rho_2 = i_2 +1 \geq 1$, $M = p+1$, and $m = s+1$, together with the restriction $1 \leq s \leq \min\{p, l\}$; in step three, we use  \eqref{weighted-seminorm} and  \eqref{coefficients 2D} to build up the link between the  derivatives of $u$ and coefficients $a_{i_1 i_2}$ and  in the last step, we use \eqref{Gamma}.  

Using the same techniques, we can derive the error estimate  for the $H^1$--seminorm.  We have
\begin{align}  \label{partial_1 2D first}
 \norm{\partial_1(\pi_{{\cal Q}_p} u -\pi_{{\cal S}_p} u)}{L^2(\hat{\k})}^2 &\leq \norm{ \partial_1(\pi_{{\cal Q}_p} u - \pi_{{\cal S}_p} u)W_2^{-1} }{L^2(\hat{\k})}^2 \nno  \\
&\leq 
\sum_{|\alpha| = s-1}  \hspace{0.2cm} \sum_{\substack { |i| = p-1, {i} \geq \alpha  \\ 
i_k\geq 1, k=1,2 }}^\infty 
 |a_{i_1 i_2}|^2  \frac{1}{{i_2}({i_2}+1)}
 \prod_{k=1}^2 \frac{2}{2{i_k}+1}  .
\end{align}
In the last step, we enlarge the summation index sets by adding the high order internal moment basis functions with coefficients $a_{i_1i_2}$, $i_k\geq 1$ for $k=1,2$ and $|i|\geq p-1$.  Thus, we have 
\begin{align} \label{partial_1 2D}
 \norm{\partial_1(\pi_{{\cal Q}_p} u - \pi_{{\cal S}_p} u)}{L^2(\hat{\k})}^2  
& \leq 
\sum_{|\alpha| = s-1}\hspace{0.2cm}  \sum_{\substack { |i| = p-1, {i} \geq \alpha  \\ 
i_k\geq 1, k=1,2 }}^\infty 
 |a_{i_1 i_2}|^2
 \Big( \prod_{k=1}^2 \frac{2}{2{i_k}+1}  \frac{\Gamma(i_k + \alpha_k+1)}{\Gamma(i_k - \alpha_k+1)}    \Big)  \nno  \\
& \times \Big( \frac{\Gamma(i_1 -\alpha_1+1) }{\Gamma(i_1+\alpha_1+1)} \frac{\Gamma(i_2-\alpha_2+1) }{\Gamma(i_2+\alpha_2+3)}   \Big)  \times 6 \nno \\
& \leq 
\sum_{|\alpha| = s-1} \hspace{0.2cm} \sum_{\substack { |i| = p-1, {i} \geq \alpha }}^\infty 
 |a_{i_1 i_2}|^2
 \Big( \prod_{k=1}^2 \frac{2}{2{i_k}+1}  \frac{\Gamma(i_k + \alpha_k+1)}{\Gamma(i_k - \alpha_k+1)}    \Big)   \nno  \\
& \times \Big( \frac{\Gamma(i_1 -\alpha_1+1) }{\Gamma(i_1+\alpha_1+1)} \frac{\Gamma(i_2-\alpha_2+1) }{\Gamma(i_2+\alpha_2+3)}   \Big)  \times 6  \nno \\
&\leq 6
 \Phi_2(p,s)
 \sum_{|\alpha| = s-1}  \norm{ W^{\alpha} D^{ \alpha} (\partial_1\partial_2u)}{L^2(\hat{\k})}^2  \nno  \\
& = 6 
 \Phi_2(p,s)
|\partial_1\partial_2 u |^2_{V^{s-1}(\hat{\k})} \leq C(s)\Big(  \frac{2}{p}\Big)^{2s}|u |^2_{H^{s+1}(\hat{\k})},
\end{align}
where in step two, we enlarge the index set again;  in step three we use Lemma \ref{optimization}, taking  $\xi_1 = \alpha_1 \geq 0$, $\xi_2 = \alpha_2+1 \geq 1$, $\rho_1 = i_1  \geq 0$, $\rho_2 = i_2 +1\geq 1$, $M = p$, and $m = s$, together with the restriction $1 \leq s \leq \min\{p, l\}$. 

Therefore, we have the bound
\begin{align} \label{relation2-H1}
\norm{\nabla(\pi_{{\cal Q}_p} u - \pi_{{\cal S}_p} u)}{L^2(\hat{\k})}^2  &\leq 12 
 \Phi_2(p,s)
|\partial_1\partial_2 u |^2_{V^{s-1}(\hat{\k})}  
\leq C(s)\Big(  \frac{2}{p}\Big)^{2s}|u |^2_{H^{s+1}(\hat{\k})}.  
\end{align}
Finally, using  \eqref{2D Q-S step 2}, \eqref{relation2-H1} and Lemma  \ref{H^1 projector Q basis}, the bounds \eqref{H1-projector-L2 norm S} and \eqref{H1-projector-H1 norm S} follow.
\end{proof}


\subsection{The $H^1$-projection operator on  the reference cube}\label{3D H1}

In this section, we shall  consider the $H^1$-projection operator over the reference cube $\hat{\k}:=(-1,1)^3$. First, we introduce the $3$D  serendipity finite element space.  

A simple way to define  the serendipity  basis is to consider a decomposition of the $C^0$ finite element space with ${\cal Q}_p$ basis over $\hat{\k}$.  For polynomial order $p$, the ${\cal S}_{p}$ basis  has  the same number of nodal basis functions and edge basis functions  as the ${\cal Q}_{p}$ basis, but the ${\cal S}_{p}$  basis only has face basis functions (those with  zero value on twelve edges and eight vertices) and  internal moment  basis functions  (those with  zero value along the element boundary $\partial \hat{\k}$) whose total degree is less than or equal  $p$.  The number of basis functions of ${\cal S}_p$ basis is calculated in  the following way
\begin{equation} \label{serendipity space 3D}
Dof ({\cal S}_p(\hat{\k})): = 8 + 12\times(p-1)+6\times\frac{(p-2)(p-3)}{2}+\frac{(p-3)(p-4)(p-5)}{6} ,
\end{equation}
here, we note that  for $p=1$, the serendipity basis only  contains  $8$ nodal basis functions  and  ${\cal S}_1(\hat{\k}):= {\cal Q}_1(\hat{\k})$. For $p\geq2$, the serendipity basis contains $(p-1)$ edge basis functions for each of the  $12$ edges. For $p\geq4$, the serendipity basis contains $(p-2)(p-3)/{2}$ face basis functions for each of the  $6$ faces. For $p\geq6$, the serendipity basis  contains ${(p-3)(p-4)(p-5)}/{6}$ internal moment basis functions.

Similarly to the $2$D case, we use $\pi_{{\cal Q}_p}:=\pi_p^{(1)}\pi_p^{(2)}\pi_p^{(3)}$ to denote the  $H^1$-projection  onto the ${\cal Q}_p$ basis.  The $H^1$-projection onto the ${\cal S}_p$ basis is denoted  by   $\pi_{{\cal S}_p}$. Additionally, we introduce some new notation for the forthcoming analysis. The projection $\pi_{{\cal S}_p}^{(1,2)}$ shall denote  the $H^1$-projection onto the serendipity spaces ${\cal S}_{p}$ with variables $(x_1, x_2)$ only, and the projections  $\pi_{{\cal S}_p}^{(1,3)}$ and $\pi_{{\cal S}_p}^{(2,3)}$ are defined in an analogous manner.

First, we explicitly construct the three-dimensional projection $\pi_{{\cal Q}_p}= \pi_p^{(1)}\pi_p^{(2)}\pi_p^{(3)}$.   For  $u \in    H^l (\hat{\k})$, $l\geq 3$,  the projection $\pi_{{\cal Q}_p}u \in {\cal Q}_p(\hat{\k})$, $p\geq 1$, is defined by
\begin{align*}
& \pi_{{\cal Q}_p} u (x_1,x_2,x_3)  :=  \int_{-1}^{x_1}\int_{-1}^{x_2} \int_{-1}^{x_3} 
\Pi_{{\cal Q}_{p-1}} \partial_1  \partial_2 \partial_3 u(x_1,x_2,x_3) \ud x_1  \ud x_2 \ud x_3 \\
& + \int_{-1}^{x_1}\int_{-1}^{x_2}   \Pi_{p-1}^{(1)}\Pi_{p-1}^{(2)}  \partial_1\partial_2  u(x_1,x_2,-1)  \ud x_1 \ud x_2  \\
& + \int_{-1}^{x_1}\int_{-1}^{x_3}    \Pi_{p-1}^{(1)}\Pi_{p-1}^{(3)}  \partial_1\partial_3  u(x_1,-1,x_3)  \ud x_1 \ud x_3  \\
&   +\int_{-1}^{x_2}\int_{-1}^{x_3}   \Pi_{p-1}^{(2)}\Pi_{p-1}^{(3)}   \partial_2\partial_3  u(-1,x_2,x_3)  \ud x_2 \ud x_3  
  +  \int_{-1}^{x_1}   \Pi_{p-1}^{(1)} \partial_1 u(x_1,-1,-1) \ud x_1  \\
&  +     \int_{-1}^{x_2}    \Pi_{p-1}^{(2)} \partial_2 u(-1,x_2,-1) \ud x_2 
 +      \int_{-1}^{x_3} \Pi_{p-1}^{(3)} \partial_3 u(-1,-1,x_3) \ud x_3  
+u(-1,-1,-1).  
\end{align*}
Then, the following Legendre polynomial expansion holds:
\begin{align}   \label{def:H^1 projector Q basis 3D full}
& \pi_{{\cal Q}_p} u(x_1,x_2,x_3)  := 
\sum_{i_1=0}^{p-1}\sum_{i_2=0}^{p-1} \sum_{i_3=0}^{p-1} a_{i_1 i_2i_3}\psi_{i_1} (x_1)\psi_{i_2}(x_2)\psi_{i_3}(x_3)   +u(-1,-1,-1)
 \nonumber \\
& 
+ \sum_{{i_1}=0}^{p-1} \sum_{{i_2}=0}^{p-1} b_{i_1i_2}\psi_{i_1}(x_1) \psi_{i_2}(x_2)
+ \sum_{{i_1}=0}^{p-1} \sum_{{i_3}=0}^{p-1} c_{i_1i_3}\psi_{i_1}(x_1) \psi_{i_3}(x_3)  
 \nonumber \\
 & 
  +\sum_{{i_2}=0}^{p-1} \sum_{{i_3}=0}^{p-1} d_{i_2i_3}\psi_{i_2}(x_2) \psi_{i_3}(x_3) 
 + \sum_{i_1=0}^{p-1} e_{i_1} \psi_{i_1}(x_1)
   +  \sum_{i_2=0}^{p-1} f_{i_2} \psi_{i_2}(x_2)
  +  \sum_{i_3=0}^{p-1} g_{i_3} \psi_{i_3}(x_3),
\end{align}
with coefficients  $a_{{i_1}{i_2}{i_3}}$, $b_{i_1 i_2}$,  $c_{i_1 i_3}$, $d_{i_2 i_3}$, give by \begin{align} \label{coefficients 3D}
a_{{i_1}{i_2}{i_3}} = \frac{2{i_1}+1}{2}\frac{2{i_2}+1}{2}\frac{2{i_3}+1}{2}\int_{\hat{\k}}    \partial_1  \partial_2 \partial_3 u(x_1,x_2,x_3) L_{i_1}(x_1) L_{i_2}(x_2)L_{i_3}(x_3)\ud x ,  \nno \\
b_{{i_1}{i_2}} = \frac{2{i_1}+1}{2}\frac{2{i_2}+1}{2} \int_{-1}^1 \int_{-1}^1
    \partial_1 \partial_2  u(x_1,x_2,-1) L_{i_1}(x_1) L_{i_2}(x_2) \ud x_1  \ud x_2 , \nno \\
c_{i_1 i_3} =  \frac{2{i_1}+1}{2} \frac{2{i_3}+1}{2} \int_{-1}^1 \int_{-1}^1    \partial_1 \partial_3  u(x_1,-1,x_3) L_{i_1}(x_1) L_{i_3}(x_3)\ud x_1 \ud x_3,  \nno \\
d_{i_2 i_3} = \frac{2{i_2}+1}{2} \frac{2{i_3}+1}{2} \int_{-1}^1 \int_{-1}^1    \partial_2 \partial_3  u(-1,x_2,x_3) L_{i_2}(x_2) L_{i_3}(x_3)\ud x_2 \ud x_3, 
\end{align}
together with  $e_{i_1}$, $ f_{i_2}$ and $ g_{i_3}$
\begin{align}\label{coefficients 3D part 2}
e_{i_1} =&\frac{2{i_1}+1}{2}  \int_{-1}^1  \partial_1   u(x_1,-1,-1) L_{i_1}(x_1)  \ud x_1, \nno \\
 f_{i_2} = &    \frac{2{i_2}+1}{2}  \int_{-1}^1  \partial_2   u(-1,x_2,-1) L_{i_2}(x_2)  \ud x_2, \nno \\
 g_{i_3} =&  \frac{2{i_3}+1}{2}  \int_{-1}^1  \partial_3   u(-1,-1,x_3) L_{i_3}(x_3)  \ud x_3. 
\end{align}


Now, we separate the face basis functions and  internal moment  basis functions from \eqref{def:H^1 projector Q basis 3D full}.  
\begin{align}  \label{def:H^1 projector Q basis 3D}
\pi_{{\cal Q}_p} u (x_1,x_2,x_3)
&:=   \sum_{i_1=1}^{p-1}\sum_{i_2=1}^{p-1} \sum_{i_3=1}^{p-1} a_{i_1 i_2i_3}\psi_{i_1} (x_1)\psi_{i_2}(x_2)\psi_{i_3}(x_3)  \nno \\
&\quad 
+ \sum_{{i_1}=1}^{p-1} \sum_{{i_2}=1}^{p-1}
\big(
a_{i_1 i_2 0}\psi_{i_1}(x_1) \psi_{i_2}(x_2)\psi_{0}(x_3)  +b_{i_1i_2} \psi_{i_1}(x_1) \psi_{i_2}(x_2) 
\big)
\nno \\
&\quad 
+ \sum_{{i_1}=1}^{p-1} \sum_{{i_3}=1}^{p-1}
\big(
a_{i_10 i_3 }\psi_{i_1}(x_1)\psi_{0}(x_2) \psi_{i_3}(x_3)  +c_{i_1i_3} \psi_{i_1}(x_1) \psi_{i_3}(x_3)
\big)    \nno  \\
&\quad 
+ \sum_{{i_2}=1}^{p-1} \sum_{{i_3}=1}^{p-1}
\big(
a_{0i_2 i_3 }\psi_{0}(x_1)\psi_{i_2}(x_2) \psi_{i_3}(x_3)   +d_{i_2i_3} \psi_{i_2}(x_2) \psi_{i_3}(x_3) 
\big)  \nno \\
&\quad + \text{edge basis} + \text{nodal basis}.
\end{align} 
Here, the first triple summation terms contains all the internal moment  basis functions only. Three double summation terms contain all the face basis functions. The edge basis functions and nodal basis functions will not be written explicitly because they play no role in the analysis. 

From the definition of ${\cal S}_p$, $\pi_{{\cal S}_p}u$ can be constructed by removing the face basis functions and  internal moment  basis functions with polynomial order greater than $p$  in $\pi_{{\cal Q}_p}u$. More specifically, $\pi_{{\cal S}_p} u \in {\cal S}_p(\hat{\k})$, $p\geq 6$, is defined by
\begin{align} \label{def:H^1 projector S basis 3D}
\pi_{{\cal S}_p}u (x_1,x_2,x_3)&:= 
  \sum_{\substack {  {|i|} = 3 \\{i_k}\geq1,k=1,2,3 }}^{p-3}
   a_{i_1 i_2i_3}\psi_{i_1} (x_1)\psi_{i_2}(x_2)\psi_{i_3}(x_3) 
     \nno \\
&\quad 
+  \sum_{\substack { {i_1+i_2} = 2 \\ {i_1}\geq1, {i_2}\geq1 }}^{p-2}
\big(
a_{i_1 i_2 0}\psi_{i_1}(x_1) \psi_{i_2}(x_2)\psi_{0}(x_3)  +b_{i_1i_2} \psi_{i_1}(x_1) \psi_{i_2}(x_2) 
\big) \nno \\
&\quad 
+\sum_{\substack {  {i_1+i_3} = 2 \\{i_1}\geq1, {i_3}\geq1 }}^{p-2}
\big(
a_{i_10 i_3 }\psi_{i_1}(x_1)\psi_{0}(x_2) \psi_{i_3}(x_3)  +c_{i_1i_3} \psi_{i_1}(x_1) \psi_{i_3}(x_3)
\big) 
 \nno  \\
&\quad
+ \sum_{\substack {{i_2+i_3} \geq 2 \\ {i_2}\geq1, {i_3}\geq1 }}^{p-2}
\big(
a_{0i_2 i_3 }\psi_{0}(x_1)\psi_{i_2}(x_2) \psi_{i_3}(x_3)   +d_{i_2i_3} \psi_{i_2}(x_2) \psi_{i_3}(x_3) 
\big)  \nno \\
& \quad+ \text{edge basis} + \text{nodal basis}
\end{align} 
For $1\leq p\leq 3$, both face basis functions and internal moment  basis functions in \eqref{def:H^1 projector S basis 3D} will vanish. For $4\leq p\leq 5$,   internal moment  basis functions in \eqref{def:H^1 projector S basis 3D} will vanish. Similar to the $2$D case,  we only consider the  $H^1$-projection $\pi_{{\cal S}_p}$ for $p\geq 6$.  


Next,  by using the stability and approximation results  for one dimensional $H^1$-projection in   \cite{hss}, we can derive the following approximation results for $\pi_{{\cal Q}_p}$. 
\begin{lemma}\label{H^1 projector Q basis 3D}
Let $\hat{\kappa}=(-1,1)^3$. Suppose that $u\in H^{l+1}(\hat{\kappa})$,
for some $l\ge 2$. Let $\pi_{{\cal Q}_p} u$ be the $H^1$-projection of $u$ onto $\mathcal{Q}_{p}(\hat{\kappa})$ with $p\geq 1$. Then, we have
\bea
\pi_{{\cal Q}_p} u = u \quad \text{at the vertices of $\hat{\k}$},
\eea
and the following error estimates hold:
\begin{align} \label{H1-projector-L2 norm 3D}
&\norm{u - \pi_{{\cal Q}_p} u}{L^2(\hat{\k})}^2
 \leq \frac{8}{p(p+1)}
\Phi_1(p,s)
\Big( \norm{\partial_1^{s+1} u}{L^2(\hat{\k})}^2
+\norm{\partial_2^{s+1} u}{L^2(\hat{\k})}^2
+\norm{\partial_3^{s+1} u}{L^2(\hat{\k})}^2 \Big) \nno \\
& 
 + \frac{8}{p^2(p+1)^2} 
\Phi_1(p,s-1)
\Big(\norm{\partial_1 \partial_2^{s} u}{L^2(\hat{\k})}^2
+     \norm{\partial_1 \partial_3^{s} u}{L^2(\hat{\k})}^2 
+     \norm{\partial_2 \partial_3^{s} u}{L^2(\hat{\k})}^2\Big)
 \nno \\
&
 + \frac{8}{p^3(p+1)^3} 
\Phi_1(p,s-2)
\norm{\partial_1 \partial_2 \partial_3^{s-1} u}{L^2(\hat{\k})}^2 
  \leq C(s)\Big(  \frac{1}{p}\Big)^{2s+2}|u |^2_{H^{s+1}(\hat{\k})}, 
\end{align}
and
\begin{align}\label{H1-projector-H1 norm 3D}
&\norm{\nabla (u - \pi_{{\cal Q}_p} u)}{L^2(\hat{\k})}^2 
\leq 
2
\Phi_1(p,s)
\Big( \norm{\partial_1^{s+1} u}{L^2(\hat{\k})}^2
+\norm{\partial_2^{s+1} u}{L^2(\hat{\k})}^2
+\norm{\partial_3^{s+1} u}{L^2(\hat{\k})}^2 \Big) \nno \\
& + \frac{8}{p(p+1)} 
\Phi_1(p,s-1)
\Big( 
\norm{\partial_1 \partial_2^s u}{L^2(\hat{\k})}^2
+\norm{\partial_2 \partial_3^{s} u}{L^2(\hat{\k})}^2 
+\norm{\partial_3\partial_1^s u}{L^2(\hat{\k})}^2
 \nno\\
&\hspace{3.5cm} 
+\norm{\partial_1 \partial_3^s u}{L^2(\hat{\k})}^2
+\norm{\partial_2 \partial_1^{s} u}{L^2(\hat{\k})}^2 
+\norm{\partial_3\partial_2^s u}{L^2(\hat{\k})}^2
\Big)
\nno \\
&
+ \frac{8}{p^2(p+1)^2} 
\Phi_1(p,s-2)
\Big(
\norm{\partial_1 \partial_2 \partial_3^{s-1} u}{L^2(\hat{\k})}^2 
+\norm{\partial_1 \partial_2 \partial_3^{s-1} u}{L^2(\hat{\k})}^2 
\nno \\  
&\hspace{3.5cm} 
+ \norm{\partial_1 \partial_2 \partial_3^{s-1} u}{L^2(\hat{\k})}^2  
\Big) \leq
C(s)\Big(  \frac{1}{p}\Big)^{2s}|u |^2_{H^{s+1}(\hat{\k})},  
\end{align} 
for any integer $s$, $2 \leq s \leq \min\{p, l\}$.
\end{lemma}

Then, we derive the $hp$-error bound for the $H^1$-projection $\pi_{{\cal S}_p}$ for $p\geq 6$.

\begin{theorem}\label{H^1 projector S basis 3D}
Let $\hat{\kappa}=(-1,1)^3$. Suppose  that $u \in H^{l+1}(\hat{\kappa})$,
for some $l\ge 2$. Let $\pi_{{\cal S}_p} u$ be the $H^1$ projection of $u$ onto $\mathcal{S}_{p}(\hat{\kappa})$ with $p\geq 6$. Then, we have
\begin{align}  \label{relation 1 S point value 3D}
\pi_{{\cal S}_p} u = u \quad \text{at the vertices of $\hat{\k}$},
\end{align} 
and for any integer $s$, $2 \leq s \leq \min\{p, l\}$,  {$p$ sufficiently large}, the following error estimates hold:
\begin{align} \label{H1-projector-L2 norm S 3D}
 \norm{u - \pi_{{\cal S}_p} u}{L^2(\hat{\k})}^2 &\leq
2 \norm{u - \pi_{{\cal Q}_p} u}{L^2(\hat{\k})}^2 + 2\norm{\pi_{{\cal Q}_p} u -\pi_{{\cal S}_p} u}{L^2(\hat{\k})}^2. 
\nno \\
& \leq C_1 
\Phi_3(p+1,s+1)
|u |^2_{H^{s+1}(\hat{\k})}
 \leq C(s)\Big(  \frac{3}{p+1}\Big)^{2s+2}|u |^2_{H^{s+1}(\hat{\k})},
\end{align} 
and
\begin{align} \label{H1-projector-H1 norm S 3D}
\norm{\nabla (u - \pi_{{\cal S}_p} u)}{L^2(\hat{\k})}^2 &\leq 
2\norm{\nabla (u - \pi_{{\cal Q}_p} u)}{L^2(\hat{\k})}^2
+2 \norm{\nabla (\pi_{{\cal Q}_p} - \pi_{{\cal S}_p} u)}{L^2(\hat{\k})}^2
\nno \\
& \leq C_2 
\Phi_3(p,s)
|u |^2_{H^{s+1}(\hat{\k})}
  \leq C(s)\Big(  \frac{3}{p}\Big)^{2s}|u |^2_{H^{s+1}(\hat{\k})}.
\end{align} 
Here, $C_1$ and $C_2$ are positive constants independent of $p$, $l$ and $s$.
\end{theorem}
\begin{proof}

The proof of \eqref{relation 1 S point value 3D} is similar to the two dimensional  case, by using the fact that the serendipity basis $\mathcal{S}_{p}$ differs from  $\mathcal{Q}_{p}$ only in the face basis functions and internal basis functions which have zero values at each vertex. 

 Next, we begin to prove relation \eqref{H1-projector-L2 norm S 3D}. Using  \eqref{def:H^1 projector Q basis 3D} and  \eqref{def:H^1 projector S basis 3D}, we have 
\begin{align} 
&\Big(\pi_{{\cal Q}_p} u - \pi_{{\cal S}_p} u \Big) (x_1,x_2,x_3) =
\sum_{\substack {{|i|}=  p-2 \\  p-1 \geq {i_k} \geq1, k=1,2,3   }}^{3(p-1)}
   a_{i_1 i_2i_3}\psi_{i_1} (x_1)\psi_{i_2}(x_2)\psi_{i_3}(x_3)  \nno \\
&\quad+  
\sum_{\substack { {i_1+i_2} = p-1 \\ p-1 \geq {i_1} \geq1 , p-1 \geq {i_2} \geq1  }}^{2(p-1)}
\Big(
a_{i_1 i_2 0}\psi_{i_1}(x_1) \psi_{i_2}(x_2)\psi_{0}(x_3)  +b_{i_1i_2} \psi_{i_1}(x_1) \psi_{i_2}(x_2)
\Big) \nno \\
&\quad +
\sum_{\substack { {i_1+i_3} = p-1 \\ p-1 \geq  {i_1}\geq1, p-1 \geq {i_3}\geq1 }}^{2(p-1)}
\Big(
a_{i_10 i_3 }\psi_{i_1}(x_1) \psi_{0}(x_2) \psi_{i_3}(x_3) +c_{i_1i_3} \psi_{i_1}(x_1) \psi_{i_3}(x_3)
\Big)    \nno  \\
&\quad+
 \sum_{\substack { {i_2+i_3} =  p-1 \\p-1  \geq {i_2} \geq1,  p-1  \geq {i_3}\geq1 }}^{2(p-1)}
\Big(
a_{0i_2 i_3 }\psi_{0}(x_1)\psi_{i_2}(x_2) \psi_{i_3}(x_3)  +d_{i_2i_3} \psi_{i_2}(x_2) \psi_{i_3}(x_3)
\Big) \nno \\
& =  T_1+T_2+T_3+T_4 .
\end{align} 
We note that the term $T_1$ only contains the internal moment  basis functions, and the  three other terms only contain the face basis functions. By using the orthogonality relation \eqref{H1 orthogonal-1} of $\psi_j(x)$ for $j\geq 1$ and $2 \leq s \leq \min\{p, l\}$, we have 
\begin{align}  \label{T1 first}
\norm{ T_1}{L^2(\hat{\k})}^2 \leq&
 \norm{ (T_1) W^{-1} }{L^2(\hat{\k})}^2  
=  \sum_{\substack {  |i| = p-2 \\p-1 \geq {i_k} \geq1,k=1,2,3 }}^{3(p-1)} 
|a_{{i_1}{i_2}{i_3}}|^2 \prod_{k=1}^3 \frac{2}{2{i_k}+1} \frac{1}{{i_k}({i_k}+1)} \nno \\
&\leq 
\sum_{|\alpha| = s-2} \hspace{0.2cm} \sum_{\substack { |i| = p-2, {i} \geq \alpha \\
{i_k} \geq1,k=1,2,3 }}^\infty
 |a_{{i_1}{i_2}{i_3}}|^2 
\prod_{k=1}^3 \frac{2}{2{i_k}+1} \frac{1}{{i_k}({i_k}+1)} \nno \\
&=
\sum_{|\alpha| = s-2} \hspace{0.2cm} \sum_{\substack { |i| = p-2, {i} \geq \alpha \\
{i_k} \geq1,k=1,2,3 }}^\infty
 |a_{{i_1}{i_2}{i_3}}|^2   
 \Big( \prod_{k=1}^3 \frac{2}{2{i_k}+1}  \frac{\Gamma(i_k + \alpha_k+1)}{\Gamma(i_k - \alpha_k+1)}    \Big)      \nno \\
& \times \Big( \prod_{k=1}^3   \frac{1}{{i_k}({i_k}+1)} \frac{\Gamma(i_k-\alpha_k+1) }{\Gamma(i_k+\alpha_k+1)}   \Big)  \nno \\
& \leq 
\sum_{|\alpha| = s-2} \hspace{0.2cm} \sum_{\substack { |i|  = p-2, {i} \geq \alpha \\
{i_k} \geq1,k=1,2,3 }}^\infty
|a_{i_1 i_2 i_3}|^2
 \Big( \prod_{k=1}^3 \frac{2}{2{i_k}+1}  \frac{\Gamma(i_k + \alpha_k+1)}{\Gamma(i_k - \alpha_k+1)}    \Big)    \nno \\
& \times \Big( \prod_{k=1}^3   \frac{\Gamma(i_k-\alpha_k+1) }{\Gamma(i_k+\alpha_k+3)}   \Big) \times 6^3  ,
\end{align} 
where in  step two, we enlarged the summation index sets by adding the high order internal moment  basis functions with coefficients $a_{i_1i_2i_3}$, $i_k\geq 1$ for $k=1,2,3$ and $|i|\geq p-2$; in the last step, we used the relation $\frac{1}{i_k (i_k+1)} \leq \frac{6}{(i_k+\alpha_k+1)(i_k+\alpha_k+2)}$ for internal moment  basis functions with $i_k\geq 1$ and $i_k \geq \alpha_k$, $k=1,2,3$. Thus, we have 
\begin{align}  \label{T1}
\norm{T_1}{L^2(\hat{\k})}^2 
& \leq 
\sum_{|\alpha| = s-2}  \hspace{0.2cm} \sum_{\substack { |i| = p-2, {i} \geq \alpha}}^\infty
|a_{i_1 i_2 i_3}|^2
 \Big( \prod_{k=1}^3 \frac{2}{2{i_k}+1}  \frac{\Gamma(i_k + \alpha_k+1)}{\Gamma(i_k - \alpha_k+1)}    \Big)   \nno \\
& \times \Big( \prod_{k=1}^3   \frac{\Gamma(i_k-\alpha_k+1) }{\Gamma(i_k+\alpha_k+3)}   \Big) \times 6^3  \nno \\
&\leq 6^3
\Phi_3(p+1,s+1)
 \sum_{|\alpha| = s-2}  \norm{ W^{\alpha} D^{ \alpha} (\partial_1\partial_2 \partial_3 u)}{L^2(\hat{\k})}^2  \nno \\
& = 216 
\Phi_3(p+1,s+1)
 |\partial_1\partial_2 \partial_3 u |^2_{V^{s-2}(\hat{\k})} 
\leq  \nno \\ 
& \leq C(s)\Big(  \frac{3}{p+1}\Big)^{2s+2}|u |^2_{H^{s+1}(\hat{\k})},
\end{align} 
where in step one,  we enlarged the summation index set by adding functions  with coefficients $a_{i_1i_2i_3}$ whose index satisfying the relation $|i|\geq p-2$, $\prod_{k=1}^3 i_k =0$; 
 in step two, we used  Lemma \ref{optimization}, with $\xi_k =  \alpha_k +1 \geq 1$, $\rho_k = i_k + 1 \geq 1$, $k=1,2,3$, together with  $M = p+1$,  $m = s+1$, and  the restriction $2 \leq s \leq \min\{p, l\}$  and in the last step, we used the relation \eqref{Gamma}.  

Next, we will derive the error bound for the term $T_2$. We first rewrite $T_2$ in the following way by adding and subtracting the same terms 
\begin{align} \label{trace term 3D}
T_2 &=  \sum_{\substack {   {i_1+i_2} = p-1 \\ p-1 \geq {i_1} \geq1 , p-1 \geq {i_2} \geq1  }}^{2(p-1)}\Big( b_{i_1i_2}\psi_{i_1}(x_1) \psi_{i_2}(x_2)   
 + 
\sum_{i_3=0}^\infty
a_{i_1 i_2 i_3}  \psi_{i_1}(x_1) \psi_{i_2}(x_2)\psi_{i_3}(x_3) \Big) \nno \\
& \qquad  - 
\sum_{\substack { {i_1+i_2} = p-1 \\ p-1 \geq {i_1} \geq1 , p-1 \geq {i_2} \geq1  }}^{2(p-1)}
\sum_{i_3=1}^\infty
a_{i_1 i_2 i_3}  \psi_{i_1}(x_1) \psi_{i_2}(x_2)\psi_{i_3}(x_3)
 = T_{2,a}-T_{2,b}.
\end{align} 
The key observation is that $T_{2,a} = \pi_p^{(1)}\pi_p^{(2)}u -  \pi_{{\cal S}_p}^{(1,2)}u$, see Appendix.  By using the $2$D approximation results \eqref{2D Q-S step 2}, we have the following bound
\begin{align} \label{relation1-H1 3D T2a}
\norm{\pi_p^{(1)}\pi_p^{(2)}u -  \pi_{{\cal S}_p}^{(1,2)}u}{L^2(\hat{\k})}^2&\leq
36
\Phi_2(p+1,s+1)
 \sum_{|\alpha| = s-1, \alpha_3=0}  \norm{ W^{\alpha} D^{ \alpha} (\partial_1\partial_2u)}{L^2(\hat{\k})}^2  \nno \\
& \leq C(s)\Big(  \frac{2}{p+1}\Big)^{2s+2}|u |^2_{H^{s+1}(\hat{\k})}.
\end{align} 
We note that $T_{2,b}$ only contains the internal moment  basis functions, so using \eqref{weight} and  the  orthogonality relation \eqref{H1 orthogonal-1} of $\psi_j(x)$ for $j\geq 1$, we have 
\begin{align}  \label{T2b}
 \norm{ T_{2,b}}{L^2(\hat{\k})}^2 &  \leq \norm{ (T_{2,b}) W^{-1} }{L^2(\hat{\k})}^2  
= \hspace{-0.5cm}
 \sum_{\substack { {i_1+i_2} = p-1 \\ p-1 \geq {i_1} \geq1 , p-1 \geq {i_2} \geq1  }}^{2(p-1)}
\sum_{i_3=1}^\infty
 |a_{{i_1}{i_2}{i_3}}|^2 \prod_{k=1}^3 \frac{2}{2{i_k}+1} \frac{1}{{i_k}({i_k}+1)}    \nno   \\
&  \leq   \sum_{\substack { |i| = p-2,\\ i_k\geq1, k=1,2,3  }}^\infty
 |a_{{i_1}{i_2}{i_3}}|^2 \prod_{k=1}^3 \frac{2}{2{i_k}+1} \frac{1}{{i_k}({i_k}+1)}  \nno   \\
& \leq \sum_{|\alpha| = s-2} \hspace{0.2cm} \sum_{\substack { |i| = p-2, {i} \geq \alpha \\
 i_k\geq1, k=1,2,3 }}^\infty
 |a_{{i_1}{i_2}{i_3}}|^2 
\prod_{k=1}^3 \frac{2}{2{i_k}+1} \frac{1}{{i_k}({i_k}+1)}, 
\end{align} 
where in step two, we used  the fact that the multi-index set for $T_{2,b}$ is a subset of multi-index $i$ with $|i|\geq p-2$ and $i_k\geq 1$, $k=1,2,3$ and then, we enlarged the index set of the summations by adding the high order internal moment  basis functions with coefficients $a_{i_1i_2i_3}$, $i_k\geq 1$ for $k=1,2,3$ and $|i|\geq p-2$.  By using \label{T1 first} and \eqref{T1}, the following bound holds
\begin{align} \label{relation1-H1 3D T2b}
\norm{ T_{2,b}}{L^2(\hat{\k})}^2 &\leq 216 
\Phi_3(p+1,s+1)
|\partial_1\partial_2 \partial_3 u |^2_{V^{s-2}(\hat{\k})} 
\leq C(s)\Big(  \frac{3}{p+1}\Big)^{2s+2}|u |^2_{H^{s+1}(\hat{\k})}. 
\end{align}
Combining  \eqref{relation1-H1 3D T2a} and  \eqref{relation1-H1 3D T2b} together with the asymptotic relation \eqref{maxi-value-asymptotic}, the following bound for  $T_{2}$ holds
\begin{align} \label{relation1-H1 3D T2}
 \norm{ T_{2}}{L^2(\hat{\k})}^2 &\leq 504 
\Phi_3(p+1,s+1)
|u |^2_{H^{s+1}(\hat{\k})} 
\leq C(s)\Big(  \frac{3}{p+1}\Big)^{2s+2} |u |^2_{H^{s+1}(\hat{\k})}. 
\end{align}
It is easy to verify that $T_3$ and $T_4$ satisfy the same error bound as  $T_2$, and thus  the $L^2$--norm error bound \eqref{H1-projector-L2 norm S 3D} is proved.

The error estimate  for the $H^1$--seminorm  \eqref{H1-projector-H1 norm S 3D} can be derived by using the same techniques used in deriving  \eqref{H1-projector-L2 norm S 3D}. As such, we have 
\begin{align} \label{T1 parital_1 first}
\norm{\partial_1 T_1}{L^2(\hat{\k})}^2 
& \leq \norm{(\partial_1 T_1) W_2^{-1}W_3^{-1}}{L^2(\hat{\k})}^2 \nno \\
&\leq
\sum_{|\alpha| = s-2}  \hspace{0.2cm}
\sum_{\substack { |i| = p-2, {i} \geq \alpha \\
 i_k\geq1, k=1,2,3  }}^\infty 
 |a_{{i_1}{i_2}{i_3}}|^2   
 \Big( \prod_{k=1}^3 \frac{2}{2{i_k}+1}  \frac{\Gamma(i_k + \alpha_k+1)}{\Gamma(i_k - \alpha_k+1)}    \Big)  \nno \\
 &\quad  \times 
\Big( \frac{\Gamma(i_1-\alpha_1+1) }{\Gamma(i_1+\alpha_1+1)} \Big) 
\Big( \prod_{k=2}^3   \frac{1}{{i_k}({i_k}+1)} \frac{\Gamma(i_k-\alpha_k+1) }{\Gamma(i_k+\alpha_k+1)}   \Big)  \nno \\
& \leq \sum_{|\alpha| = s-2}  \hspace{0.2cm}
\sum_{\substack { |i| = p-2, {i} \geq \alpha \\
 i_k\geq1, k=1,2,3  }}^\infty 
 |a_{i_1 i_2 i_3}|^2
 \Big( \prod_{k=1}^3 \frac{2}{2{i_k}+1}  \frac{\Gamma(i_k + \alpha_k+1)}{\Gamma(i_k - \alpha_k+1)}    \Big)  \nno \\
&\quad  \times 6^2 
\Big( \frac{\Gamma(i_1-\alpha_1+1) }{\Gamma(i_1+\alpha_1+1)} \Big) 
\Big( \prod_{k=2}^3   \frac{\Gamma(i_k-\alpha_k+1) }{\Gamma(i_k+\alpha_k+3)}   \Big) ,
\end{align}
Now, we have
\begin{align} \label{T1 parital_1}
\norm{\partial_1 T_1}{L^2(\hat{\k})}^2 
& \leq \sum_{|\alpha| = s-2}  \hspace{0.2cm} \sum_{\substack { |i| = p-2, {i} \geq \alpha   }}^\infty |a_{i_1 i_2 i_3}|^2
 \Big( \prod_{k=1}^3 \frac{2}{2{i_k}+1}  \frac{\Gamma(i_k + \alpha_k+1)}{\Gamma(i_k - \alpha_k+1)}    \Big)  \nno \\
&\quad  \times 6^2 
\Big( \frac{\Gamma(i_1-\alpha_1+1) }{\Gamma(i_1+\alpha_1+1)} \Big) 
\Big( \prod_{k=2}^3   \frac{\Gamma(i_k-\alpha_k+1) }{\Gamma(i_k+\alpha_k+3)}   \Big)  \nno \\
&\leq 6^2
\Phi_3(p,s)
\sum_{|\alpha| = s-2}  \norm{ W^{\alpha} D^{ \alpha} (\partial_1\partial_2 \partial_3 u)}{L^2(\hat{\k})}^2  \nno \\
& = 36 
\Phi_3(p,s)
|\partial_1\partial_2 \partial_3 u |^2_{V^{s-2}(\hat{\k})} 
\leq C(s)\Big(  \frac{3}{p}\Big)^{2s}|u |^2_{H^{s+1}(\hat{\k})},
\end{align}
where in step one, we enlarged the summation index set by adding functions with coefficients $a_{i_1i_2i_3}$ whose index satisfying the relation $|i|\geq p-2$, $\prod_{k=1}^3 i_k =0$;
 in step two we use Lemma \ref{optimization}, taking  $\xi_1 = \alpha_1 \geq 0$, $\xi_2 = \alpha_2+1 \geq 1$, $\xi_3 = \alpha_3+1 \geq 1$, $\rho_1 = i_1  \geq 0$, $\rho_2 = i_2 +1\geq 1$, $\rho_3 = i_3+1\geq 1$, $M = p$, and $m = s$, together with the restriction $2 \leq s \leq \min\{p, l\}$.

Next, we consider the error bound for term $T_{2}$. By using  \eqref{partial_1 2D}, the following bound holds for $T_{2,a}$
\begin{align}\label{T2a partial1}
\norm{ \partial_1 (T_{2,a}) }{L^2(\hat{\k})}^2   &= \norm{ \partial_1 (\pi_p^{(1)}\pi_p^{(2)}u -  \pi_{{\cal S}_p}^{(1,2)}u) }{L^2(\hat{\k})}^2  \nno \\
&\leq
6
\Phi_2(p,s)\hspace{-0.2cm}
\sum_{|\alpha| = s-1, \alpha_3=0} \hspace{-0.5cm} \norm{ W^{\alpha} D^{ \alpha} (\partial_1\partial_2u)}{L^2(\hat{\k})}^2   
\leq C(s)\Big(  \frac{2}{p}\Big)^{2s}|u |^2_{H^{s+1}(\hat{\k})}.
\end{align}
With the help of \eqref{T2b} and \eqref{T1 parital_1}, the following bound holds
\begin{align} \label{T2b partial1}
\norm{\partial_1 T_{2,b}}{L^2(\hat{\k})}^2
& \leq \norm{ (\partial_1 T_{2,b}) W_2^{-1}W_3^{-1} }{L^2(\hat{\k})}^2     \nno \\
& \leq \sum_{|\alpha| = s-2}  \sum_{\substack { |i| = p-2, {i} \geq \alpha  }}^\infty
 |a_{{i_1}{i_2}{i_3}}|^2 
 \Big(  \frac{2}{2{i_1}+1}\Big)
\Big( \prod_{k=2}^3 \frac{2}{2{i_k}+1} \frac{1}{{i_k}({i_k}+1)} \Big) \nno  \\ 
& =  36 
\Phi_3(p,s)
|\partial_1\partial_2 \partial_3 u |^2_{V^{s-2}(\hat{\k})} 
\leq C(s)\Big(  \frac{3}{p}\Big)^{2s}|u |^2_{H^{s+1}(\hat{\k})}.
\end{align}
Combing \eqref{T2a partial1}, \eqref{T2b partial1} together with the asymptotic relation \eqref{maxi-value-asymptotic}, then  the  following error bound for term $T_2$.
\begin{align} \label{relation1-H1 3D T2 parital1}
\norm{\partial_1 T_{2}}{L^2(\hat{\k})}^2 \leq 84 
\Phi_3(p,s)
 |u |^2_{H^{s+1}(\hat{\k})}
\leq C(s)\Big(  \frac{3}{p}\Big)^{2s} |u |^2_{H^{s+1}(\hat{\k})}.  
\end{align}
The above error bound for $T_2$ term is also  the error bound for terms $T_3$ and $T_4$. By using \eqref{T2a partial1} and  \eqref{T2b partial1},  it is to see that the following relation holds
\begin{align}\label{relation1-H1 3D  parital1}
\norm{\partial_1 (\pi_{{\cal Q}_p} - \pi_{{\cal S}_p} u)}{L^2(\hat{\k})}^2\leq C^* 
\Phi_3(p,s)
  |u |^2_{H^{s+1}(\hat{\k})}
\leq C(s)\Big(  \frac{3}{p}\Big)^{2s} |u |^2_{H^{s+1}(\hat{\k})}, 
\end{align}
where $C^*$ is a positive constant independent of $p$, $s$ and $l$. We note that above $L^2$--norm error bound also holds for the rest two  partial derivatives. So     the  $H^1$--seminorm bound  \eqref{H1-projector-H1 norm S 3D} is proved.
 \end{proof}

\begin{remark}
We again make a comparison between the bounds in the $L^2$--norm and $H^1$--seminorm, given in Lemma \ref{H^1 projector Q basis} for $d=2$ and Lemma \ref{H^1 projector Q basis 3D} for $d=3$  respectively for  $\pi_{{\cal Q}_p}$, and Theorem \ref{H^1 projector S basis} for $d=2$ and Theorem \ref{H^1 projector S basis 3D} for $d=3$  respectively for  $\pi_{{\cal S}_p}$. Similarly to the comparisons for the $L^2$-projection onto  ${\cal P}_p$ and ${\cal Q}_p$,   both bounds are $p$-optimal in both  Sobolev regularity and  polynomial order. We can also see that the bounds for $\pi_{{\cal S}_p}$  have a larger constant than those for   $\pi_{{\cal Q}_p}$, and this constant depends on dimension $d$.  Moreover, we point out that the optimal approximation results for the $H^1$-projection with ${\cal S}_p$ basis in Theorem \ref{H^1 projector S basis} and Theorem \ref{H^1 projector S basis 3D} directly imply the  $hp$-optimal error bound for the  $L^2$--norm   on the trace of $\hat{\k}$ for $\pi_{{\cal S}_p}$.
\end{remark}

\begin{remark}
We note that in the Theorem \ref{H^1 projector S basis} and \ref{H^1 projector S basis 3D}, the minimum Sobolev regularity requirements for defining $H^1$-projection is $u\in H^d(\hat{\k})$ for the  reference element. In fact, this regularity requirement can be relaxed by using the the tensor product Sobolev spaces,  cf. \cite{schwab,thesis}. In this work,  we do not consider  the minimum regularity assumptions because we only consider  the standard Sobolev spaces.
\end{remark}

\subsection{The $H^1$-projection operator onto  the ${\cal P}_p$ basis} 

Finally, we present the error bound for  $\pi_{{\cal P}_p} $ which we shall define now. The key observation is that the  ${\cal P}_p$ basis with polynomial order $p$ contains the ${\cal S}_{p+1-d}$ basis for $p\geq d$, see \cite{arnold2011serendipity}. Then, we can simply define $\pi_{{\cal P}_p} = \pi_{{\cal S}_{p+1-d}} $ for $d=2,3$.

\begin{corollary}\label{H^1 projector P basis}
Let $\hat{\kappa}=(-1,1)^d$, $d=2,3$. Suppose  that $u\in H^{l+1}(\hat{\kappa})$,
for some $l\ge d-1$. Let  {$\pi_{{\cal P}_p} u : = \pi_{{\cal S}_{p+1-d}} u$} be the $H^1$ projection of $u$ onto $\mathcal{P}_{p}(\hat{\kappa})$ with $p\geq 3d-1$. Then, we have:  {
\begin{align}\label{relation 1 S}
\pi_{{\cal P}_p} u = u \quad \text{at the vertices of $\hat{\k}$},
\end{align}}
and the following error estimates hold:
\begin{align}\label{H1-projector-L2 norm P}
 \norm{u - \pi_{{\cal P}_p} u}{L^2(\hat{\k})}^2 = \norm{u - \pi_{{\cal S}_{p+1
-d}} u}{L^2(\hat{\k})}^2 \leq C(s)\Big(  \frac{d}{p+1-d}\Big)^{2s+2}|u |^2_{H^{s+1}(\hat{\k})}.
\end{align}
and
\begin{align}\label{H1-projector-H1 norm P}
\norm{\nabla (u - \pi_{{\cal P}_p} u)}{L^2(\hat{\k})}^2 = \norm{\nabla (u - \pi_{{\cal S}_{p+1
-d}} u)}{L^2(\hat{\k})}^2  \leq  C(s)\Big(  \frac{d}{p-d}\Big)^{2s}|u |^2_{H^{s+1}(\hat{\k})}.
\end{align}
for any integer $s$, $d-1 \leq s \leq \min\{p+1-d, l\}$,  {$p$ sufficiently large}.
\end{corollary}
\begin{remark}
We emphasize that \emph{  the  above error bound for the  {$\pi_{{\cal P}_p}$} projection is  $p$-suboptimal by one order  for $d=2$ and two orders  for $d=3$ for sufficiently smooth  functions, but it is $p$-optimal for functions with finite Sobolev regularity in the case  $l\leq p+1-d$. } However, sub-optimality by one or two orders in $p$ is better than using the $\pi^{\cal Q}_{\ujump{p/d}}$ projection, as suggested by \cite{schwab} (see Corollary 4.52 on p190), which is sub-optimal in $p$ by at least $p/2$ orders for sufficiently smooth functions for $d=2$. Moreover, the  one or two  order sub-optimality in $p$ for analytic functions does not influence the exponent of  the exponential rate of convergence, as we shall see below.
\end{remark}



\section{Exponential convergence for analytic solutions}

We shall be concerned with the  proof of  exponential convergence for serendipity FEMs and DGFEMs with the ${\cal P}_p$ basis over tensor product elements. For simplicity, we only consider the case when the given problem is piecewise analytic over the whole computational domain. Exponential convergence is then achieved by fixing the computational mesh, and increasing the polynomial order $p$. Only  parallelepiped meshes are considered, which are the affine family obtained from  the reference element $\hat{\k}=(-1,1)^d$. The analysis of FEMs and DGFEMs with a general $hp$-refinement strategy is beyond the scope of this analysis (see \cite{MR3351174,schotzau2013hp1,schotzau2013hp2,schotzau2016hp} for the analysis for both methods employing the ${\cal Q}_p$ basis).

The proof of exponential convergence for FEMs and DGFEMs depends on proving exponential convergence of $L^2$- and $H^1$-projections for piecewise analytic functions under $p$-refinement. The $H^1$-projection $\pi_{{\cal S}_p}$  onto ${\cal S}_p$ can be directly applied to $p$-FEMs for second order elliptic problems with the same optimal rate as the $H^1$ projection $\pi_{{\cal Q}_p}$, see \cite{schwab} for details.   For deriving error bounds of DGFEMs using  the $L^2$- and $H^1$- projections onto ${\cal Q}_p$, we refer to \cite{hss,newpaper,thesis}.  \emph{Following similar techniques, we can prove the corresponding  $hp$-bounds for  DGFEMs employing  the ${\cal P}_p$ basis, albeit with sub-optimal rate in $p$. The sub-optimality in $p$ is  due to the fact that the $p$-optimal bound for $L^2$-projection onto  ${\cal P}_p$ basis  over the trace  of the tensor product elements is still open. Additionally, the $H^1$-projection onto the  ${\cal P}_p$ basis  is  suboptimal in $p$ by $d-1$ orders for sufficiently smooth functions}. However, we point out that the suboptimality in $p$ by $d-1$ order, with $d=2,3$, does not influence the exponent of the exponential rate of convergence. 


 
Next, we focus on deriving the exponential convergence for  the $L^2$-projections in the $L^2$--norm and $H^1$-projections in the $L^2$--norm and $H^1$--seminorm   on  analytic problems under $p$-refinement on shape-regular $d$-parallelepiped  meshes. The extension to anisotropic meshes will be consider in the future.
 
 Let $\k$ be a parallelepiped element. For a function $u$ having an analytic extension into an open neighbourhood of $\bar{\k}$,  we have:
\begin{align} \label{analytic extension}
\exists R_\k >0, \quad  {C(u)}>0, \quad \forall s_\k : |u|_{H^{s_\k}(\kappa)} \leq  {C(u)} (R_\k)^{s_\k}\Gamma(s_\k+1) |\k|^{1/2},
\end{align}
 where $|\k|$ denotes the measure of element $\k$, cf. \cite[Theorem 1.9.3]{davis1975interpolation}.  
 \begin{lemma}\label{Exponential convergence p}
 Let $u:\k \rightarrow \mathbb{R}$ have an analytic extension to an open neighbourhood of $\bar{\k}$. Also let $p_\k\geq 0$ and $0\leq s_\k\leq p_\k+1$ be two positive numbers such that $s_\k =\epsilon (p_\k+1) $, $0\leq \epsilon \leq 1$ and $d=2,3$. Then the following bounds hold: 
\begin{align*}
 \hspace{0cm} \norm{u -  \Pi_{{\cal Q}_{p_\k}} u }{L^2({{\kappa}})}^2  \le C  ( h_\k)^{2s_\k} 
 \Phi_1(p_\k+1,s_\k)
|u|_{H^{s_\k}(\hat{\kappa})}^2  
\leq C(u)(p_\k+1) e^{-2 b_{1,\k} (p_\k+1)} |\k|, 
\end{align*}
 and
\begin{align*}
\norm{u -  \Pi_{{\cal P}_{p_\k}} u }{L^2({{\kappa}})}^2  \le   C ( {h_\k})^{2s_\k}
 \Phi_d(p_\k+1,s_\k)
|u|_{H^{s_\k}(\hat{\kappa})}^2
\leq C(u)(p_\k+1) e^{-2 b_{2,\k} (p_\k+1)} |\k|.
\end{align*}
 Here, $C$ and $C(u)$ are  positive constants  depending  elemental shape regularity, and $C(u)$ also depends on $u$.  $F_1(R_\k,\epsilon)  = \frac{(1-\epsilon)^{1-\epsilon}}{(1+\epsilon)^{1+\epsilon}} (\epsilon R_k)^{2\epsilon}$, $\epsilon_{\min} = {1}/{\sqrt{1+R_\k^2}}$, 
 $b_{1,\k}:= \frac{1}{2}|\log F_1(R_\k,\epsilon_{\min}) | + \epsilon_{\min} |\log {h_\k}| $ and  $b_{2,\k}:= b_{1,\k}-  \epsilon_{\min} \log d$.
 \end{lemma}
\begin{proof}
Using standard scaling arguments for $\k$ together with Lemma \ref{L2-Q-basis} and Theorem \ref{L2-projector-pbasis}, we have the approximation results for the  $L^2$-projection over $\k$. For brevity, we set $q_\k=p_\k+1$. By employing  the relation \eqref{Gamma}  and the fact $|u|_{V^{l}({\k})}\leq |u|_{H^{l}({\k})}$, we have the bounds:
\begin{align}
\Phi_1(p_\k+1,s_\k)
  |u|_{H^{s_\k}(\hat{\kappa})}^2 
  &\leq  {C(u)} (R_\k)^{2s_\k}\Gamma(s_\k+1)^2  \frac{\Gamma( {q_\k-s_\k}+1)}{\Gamma( {q_\k+s_\k}+1)} |\k|  \nno \\
&\leq  {C(u)} (R_\k)^{2 \epsilon q_\k} 
\frac{(\epsilon q_\k)^{2\epsilon q_\k+1}}{e^{2\epsilon q_\k}}
  \frac{((1-\epsilon)q_\k)^{(1-\epsilon)q_\k} e^{-(1-\epsilon)q_\k}}{((1+\epsilon)q_\k)^{(1+\epsilon)q_\k} e^{-(1+\epsilon)q_\k}}  |\k|  \nno \\
  &\leq  {C(u)}q_\k (F_1(R_\k,\epsilon))^{q_\k} |\k|, \nno
\end{align}
where
$$
F_1(R_\k,\epsilon)  = \frac{(1-\epsilon)^{1-\epsilon}}{(1+\epsilon)^{1+\epsilon}} (\epsilon R_k)^{2\epsilon}.
$$ 
Recalling \eqref{analytic extension}, we have  $R_\k>0$,
\begin{equation}\label{Fmin}
\min_{0<\epsilon<1} F_1(R_\k,\epsilon) =F_1(R_\k,\epsilon_{\min}) =  \left(\frac{R_\k}{\sqrt{1+R_\k^2}+1}\right)^2 < 1, \quad \epsilon_{\min} = \frac{1}{\sqrt{1+R_\k^2}}.
\end{equation}
Thus, we have 
\begin{align}
\frac{\Gamma( {p_\k-s_\k}+2)}{\Gamma( {p_\k+s_\k}+2)}  
  |u|_{H^{s_\k}(\hat{\kappa})}^2  \leq  {C(u)} q_\k e^{-|\log F_1(R_\k,\epsilon_{\min}) |q_\k} |\k|.
\end{align}
Therefore, we have the  exponential convergence for the $L^2$-projection $\Pi_{{\cal Q}_{p_\k}}$, via
\begin{align}
 \norm{u -  \Pi_{{\cal Q}_{p_\k}} u }{L^2({{\kappa}})}^2 
\leq  {C(u)} ( {p_\k}+1) e^{-2b_{1,\k} (p_\k+1)} |\k|,
\end{align}
with  $b_{1,\k}:= \frac{1}{2}|\log F_1(R_\k,\epsilon_{\min}) | + \epsilon_{\min} |\log {h_\k}| $.  Similarly,  for the $L^2$-projection $\Pi_{{\cal P}_{p_\k}}$,  Stirling's formula implies
\begin{align*}
\Phi_d(p_\k+1,s_\k)
  |u|_{H^{s_\k}(\hat{\kappa})}^2 
  &\leq  {C(u)} (R_\k)^{2s_\k}\Gamma(s_\k+1)^2 
  \Big( \frac{\Gamma( \frac{q_\k-s_\k}{d}+1)}{\Gamma( \frac{q_\k+s_\k}{d}+1)}  \Big)^d
   |\k|   \\
&\leq  {C(u)} (R_\k)^{2 \epsilon q_\k} 
\frac{(\epsilon q_\k)^{2\epsilon q_\k+1}}{e^{2\epsilon q_\k}}
  \frac{((1-\epsilon)q_\k)^{(1-\epsilon)q_\k} (ed)^{-(1-\epsilon)q_\k}}{((1+\epsilon)q_\k)^{(1+\epsilon)q_\k} (ed)^{-(1+\epsilon)q_\k}}  |\k|   \\
  &\leq  {C(u)} q_\k (F_2(R_\k,\epsilon))^{q_\k} |\k|, 
\end{align*}
where,
$$
F_2(R_\k,\epsilon)  = \frac{(1-\epsilon)^{1-\epsilon}}{(1+\epsilon)^{1+\epsilon}} (\epsilon R_kd)^{2\epsilon},
$$
with the minimum,
$$
\min_{0<\epsilon<1} F_2(R_\k,\epsilon) =  \left(\frac{R_\k d}{\sqrt{1+(R_\k d)^2}+1}\right)^2 < 1.
$$
In order to compare with the  slope of projection $\Pi_{{\cal Q}_{p_\k}}$, here we will use the same $\epsilon_{\min}$. We have 
$$
\min_{0<\epsilon<1} F_2(R_\k,\epsilon) \leq F_2(R_\k,\epsilon_{\min}) = F_1(R_\k,\epsilon_{\min}) d^{2 \epsilon_{\min} } .
$$
Thus, we have 
\begin{align}
 \norm{u -  \Pi_{{\cal P}_{p_\k}}u }{L^2({{\kappa}})}^2 
\leq  {C(u)}(p+1) e^{-2 b_{2,\k} (p_\k+1)} |\k|,
\end{align}
with  slope $b_{2,\k}:= \frac{1}{2}|\log F_1(R_\k,\epsilon_{\min}) | + \epsilon_{\min}( |\log {h_\k}| - \log d) $. 
\end{proof}

Next, we begin to derive the exponential convergence for $H^1$-projections.
 \begin{lemma}\label{Exponential convergence p H1}
 Let $u:\k \rightarrow \mathbb{R}$ have an analytic extension to an open neighbourhood of $\bar{\k}$. Also let $p_\k \geq 2d$ and $ (d-1) \leq s_\k \leq p_\k$ be two positive numbers such that $s_\k =\epsilon p_\k $, $0< \epsilon \leq 1$ and $d=2,3$. Then the following bounds hold: 
\begin{align} 
  \norm{u -  \pi_{{\cal Q}_{p_\k}} u }{L^2({{\kappa}})}^2  &\le C  
{( h_\k)^{2s_\k+2}} \Phi_1(p_\k+1,s_\k+1)
|u|_{H^{s_\k+1}(\hat{\kappa})}^2 \leq C(u)p_\k e^{-2 b_{1,\k} p_\k } |\k|, \nno
\end{align}
\begin{align} 
 \norm{u -  \pi_{{\cal S}_{p_\k}} u }{L^2({{\kappa}})}^2  &\le   C ( {h_\k})^{2s_\k+2}
 \Phi_d(p_\k+1,s_\k+1)
|u|_{H^{s_\k+1}(\hat{\kappa})}^2 \leq C(u)p_\k e^{-2 b_{2,\k} p_\k} |\k|,  \nno
\end{align}
 and
\begin{align}
\norm{\nabla (u -  \pi_{{\cal Q}_{p_\k}} u) }{L^2({{\kappa}})}^2  &\le C ( h_\k)^{2s_\k} 
 \Phi_1(p_\k,s_\k)
|u|_{H^{s_\k+1}(\hat{\kappa})}^2 \leq C(u)p_\k^{3} e^{-2 b_{1,\k} p_\k } |\k|, \nno
\end{align}
\begin{align}
\norm{\nabla (u -  \pi_{{\cal S}_{p_\k}} u) }{L^2({{\kappa}})}^2  &\le   C ( {h_\k})^{2s_\k}
 \Phi_d(p_\k,s_\k)
|u|_{H^{s_\k+1}(\hat{\kappa})}^2 \leq C(u)p_\k^{3} e^{-2 b_{2,\k} p_\k} |\k|. \nno
\end{align}
 Here, $C$ and $C(u)$ are  positive constants  depending  elemental shape regularity, and $C(u)$ also depends on $u$.  $F_1(R_\k,\epsilon)  = \frac{(1-\epsilon)^{1-\epsilon}}{(1+\epsilon)^{1+\epsilon}} (\epsilon R_k)^{2\epsilon}$, $\epsilon_{\min} = {1}/{\sqrt{1+R_\k^2}}$, 
 $b_{1,\k}:= \frac{1}{2}|\log F_1(R_\k,\epsilon_{\min}) | + \epsilon_{\min} |\log {h_\k}| $ and  $b_{2,\k}:= b_\k^1-  \epsilon_{\min} \log d$.
 \end{lemma}
 \begin{proof}
 The proof follows by the same techniques used in Lemma \ref{Exponential convergence p}.
 \end{proof}

In the above Lemma \ref{Exponential convergence p} and Lemma \ref{Exponential convergence p H1}, we can see that the  $L^2$--norm error for both $L^2$-projections $\Pi_{{\cal Q}_{p_\k}}$ and $\Pi_{{\cal P}_{p_\k}}$,  and the $L^2$--norm and $H^1$--seminorm error for the $H^1$-projections  $\pi_{{\cal S}_{p_\k}}$ and $\pi_{{\cal Q}_{p_\k}}$ decay exponentially for analytic functions under  $p$-refinement.  If we measure the error against $p$, the exponent $b_{1,\k}$ for the ${\cal Q}_p$ basis  is slightly greater than the exponent  $b_{2,\k}$ for the ${\cal P}_p$ basis and ${\cal S}_p$ basis by a small factor of $(\log d)/{\sqrt{1+R_\k^2}}$.  By using Lemma \ref{Exponential convergence p} and Lemma \ref{Exponential convergence p H1}, we can also derive the following theorem.

\begin{theorem}\label{better slope}
Let $u$ be an analytic function as defined in \eqref{analytic extension}, and exponent $b_{1,\k}$ and $b_{2,\k}$  defined in Lemma \ref{Exponential convergence p}   Then, there exists $C>0$ such that  following bounds hold:
\begin{align} \label{exponential Q dof-refine}
 \norm{u -  \Pi_{{\cal Q}_{p_\k}} u }{L^2({{\kappa}})}^2  
\leq C e^{-2 b_{1,\k} \sqrt[d]{Dof}} ,
\end{align} 
\begin{align}  \label{exponential P dof-refine}
 \norm{u -  \Pi_{{\cal P}_{p_\k}} u }{L^2({{\kappa}})}^2 
\leq Ce^{-2( b_{2,\k}  \sqrt[d]{d!}) \sqrt[d]{Dof}} ,
\end{align} 
 and 
\begin{align}  \label{exponential Q dof-refine H1 L2}
 \norm{u -  \pi_{{\cal Q}_{p_\k}} u }{L^2({{\kappa}})}^2  
\leq C e^{-2 b_{1,\k} \sqrt[d]{Dof}} ,
\end{align} 
\begin{align}  \label{exponential S dof-refine H1 L2}
 \norm{u -  \pi_{{\cal S}_{p_\k}} u }{L^2({{\kappa}})}^2 
\leq Ce^{-2(b_{2,\k}  \sqrt[d]{d!}) \sqrt[d]{Dof}} ,
\end{align} 
 and 
\begin{align}  \label{exponential Q dof-refine H1 H1}
 \norm{\nabla (u -  \pi_{{\cal Q}_{p_\k}} u) }{L^2({{\kappa}})}^2  
\leq C e^{-2 b_{1,\k} \sqrt[d]{Dof}} ,
\end{align} 
\begin{align}  \label{exponential S dof-refine H1 H1}
 \norm{ \nabla (u -  \pi_{{\cal S}_{p_\k}} u) }{L^2({{\kappa}})}^2 
\leq Ce^{-2(b_{2,\k}  \sqrt[d]{d!}) \sqrt[d]{Dof}}.
\end{align} 
\end{theorem}
\begin{proof}
By recalling the relationship between degrees of freedom and polynomial order $p$  for both the ${\cal Q}_p$ and ${\cal P}_p$ bases, we have 
\begin{align} \label{dof Q}
Dof({\cal Q}_p) = (p+1)^d,
\end{align} 
and 
\begin{align} \label{dof P}
 Dof({\cal P}_p) = {{p+d}\choose{d}} =\frac{(p+1)^d}{d!} +{\cal O}((p+1)^{d-1}).
\end{align} 
Then,  \eqref{exponential Q dof-refine} and \eqref{exponential P dof-refine} follow from   Lemma \ref{Exponential convergence p}.  

By using relations \eqref{serendipity space 2D} and \eqref{serendipity space 3D}, we have the asymptotic relation 
\begin{align} \label{dof S}
Dof({\cal S}_p) \approx  \frac{p^d}{d!} +{\cal O}(p^{d-1}).
\end{align} 
The relations \eqref{exponential Q dof-refine H1 L2}, \eqref{exponential S dof-refine H1 L2}, \eqref{exponential Q dof-refine H1 H1} and  \eqref{exponential S dof-refine H1 H1} follow from the Lemma \ref{Exponential convergence p H1}.  
\end{proof}

For $d=2, 3$, if the following condition  
\begin{align} \label{small mesh and smooth}
 \frac{1}{2}|\log F_1(R_\k,\epsilon_{\min}) |  + \epsilon_{\min} |\log {h_\k}|  \gg \epsilon_{\min}\log d,
\end{align} 
holds, then we have $b_{2,\k}  \approx b_{1,\k}$.  It is easy to see that  for  small $R_\k$ or  small mesh size $h$, the condition \eqref{small mesh and smooth} will be satisfied. Moreover, we point out that an analytic function having  sufficiently small  $R_\k$  is  equivalent to    the function having an analytic continuation into a sufficiently large open neighbourhood of $\bar{\k}$, see \cite{davis1975interpolation} for details.

 Now, if we consider the error in terms of $\sqrt[d]{Dof}$ for the above bounds, the exponent for the  exponential convergence rate of the  ${\cal P}_p$ basis and the ${\cal S}_p$ basis are larger than the exponent  for the   ${\cal Q}_p$ basis   by a fixed factor of $\sqrt[d]{d!}$.

We have observed  a steeper slope in error against $\sqrt[d]{Dof}$ for FEMs with ${\cal S}_p$ basis and DGFEMs with ${\cal P}_p$ basis. For $d=2$, this suggests a typical ratio between convergence  slopes of DGFEMs with ${\cal P}_p$ and ${\cal Q}_p$ basis, FEMs with ${\cal S}_p$ and ${\cal Q}_p$ basis,  to be $\sqrt{2!}\approx 1.414$. For $d=3$, this ratio is  $\sqrt[3]{3!}\approx 1.817$.  The numerical examples in Section \ref{Numerical examples} show that the ratio is slightly worse than the ideal ratio, but it is not far from the ideal ratio. 


\section{Numerical examples} \label{Numerical examples}

We present some numerical examples to confirm the theoretical analysis in the previous  sections.   {All the numerical examples are computed by Matlab on the High Performance Computing facility ALICE of the University of Leicester.}  For simplicity of presentation,  we use DGFEM(P) and DGFEM(Q) to denote the DGFEMs with local polynomial basis consisting of either ${\cal P}_p$ or ${\cal Q}_p$ polynomials and  use  FEM(S) and FEM(Q) to denote the FEMs with local polynomial basis consisting of either ${\cal S}_p$ or ${\cal Q}_p$ polynomials. 

The comparisons are mainly made between the slope of FEM(S) and FEM(Q)  over  square meshes for $d=2$ and hexahedral  meshes for $d=3$ under $p$-refinement.  The slopes of the convergence lines are calculated by taking the average of the last two slopes of the  line segments of each convergence line. We will also present an example comparing  DGFEM(P)  and DGFEM(Q). For more numerical examples for DGFEMs, see \cite{D17,Zhaonan}. 

%
%
%

\subsection{Example 1}
In the first example, we investigate the  computational efficiency of  DGFEM(P) and DGFEM(Q) schemes. To this end, we consider a partial differential equation with nonnegative characteristic form of mixed type. Let $\Omega=(-1,1)^2$, and consider the PDE problem:
\begin{equation}
\begin{cases}
 -x^2u_{yy} + u_x+u = 0, &\quad \text{for } -1 \leq x \leq 1, y > 0 ,\\
u_x+u = 0, &\quad \text{for } -1 \leq x \leq 1, y \leq 0 ,
\end{cases}
\end{equation}
with exact solution:
\begin{equation}
u(x,y) =
\begin{cases}
\sin (\frac{1}{2}\pi (1+y)) \exp(-(x+\frac{\pi^2x^3 }{12})),& \text{for } -1 \leq x \leq 1, y > 0 ,\\
\sin (\frac{1}{2}\pi (1+y)) \exp(-x), &\text{for } -1 \leq x \leq 1, y\leq 0.
\end{cases}
\end{equation}
This problem is hyperbolic in the region $y\leq 0$ and parabolic for $y > 0$. 
In order to ensure continuity of the normal flux across $y=0$, where the partial differential
equation changes type, the exact solution has a discontinuity across the line $y=0$,
cf. \cite{cangiani2015hp,thesis}.


\begin{figure}[!t]
\centering
\includegraphics[width=0.45\linewidth]{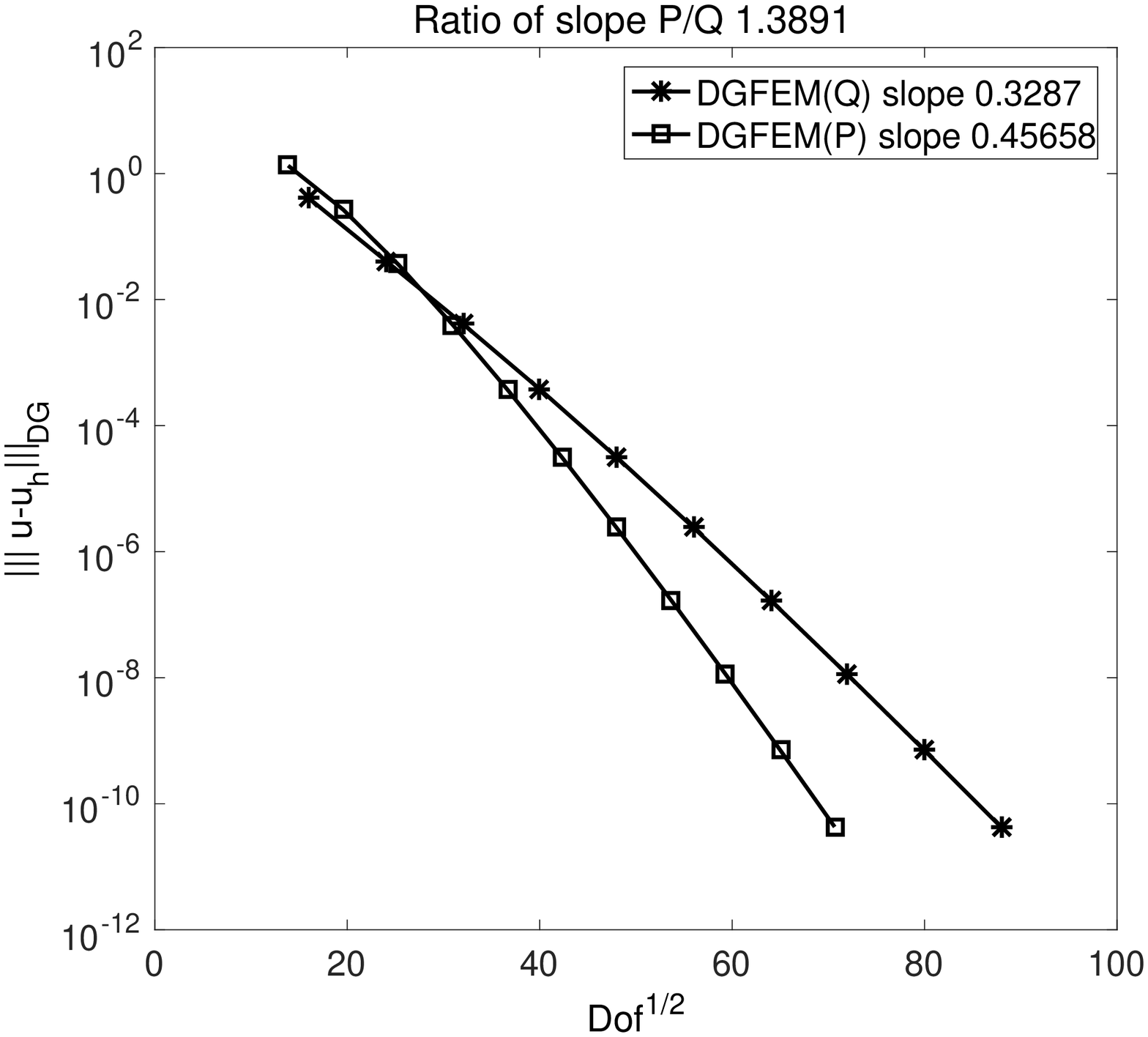} 
\hspace{0cm}
\includegraphics[width=0.45\linewidth]{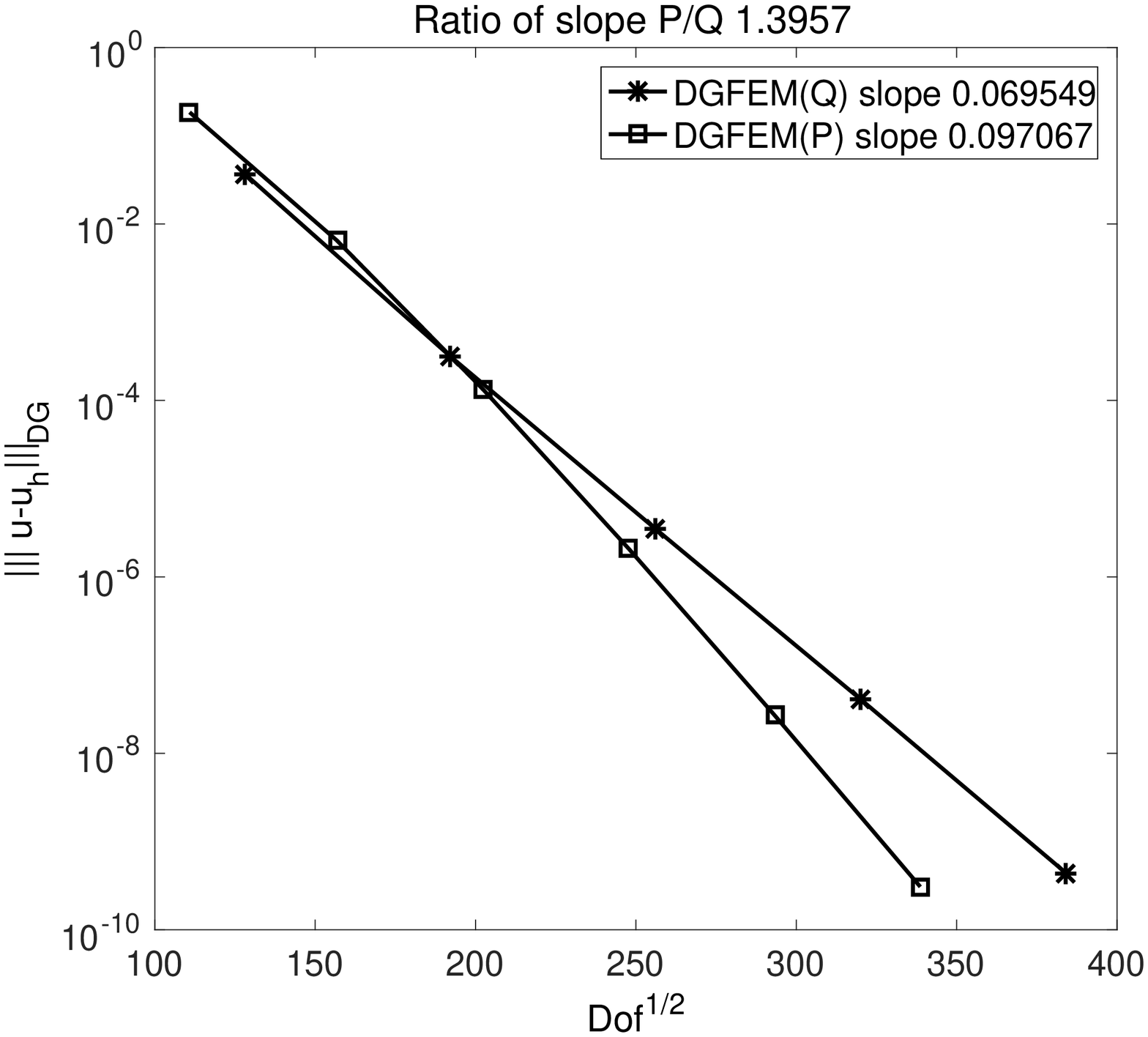}
\\
\includegraphics[width=0.45\linewidth]{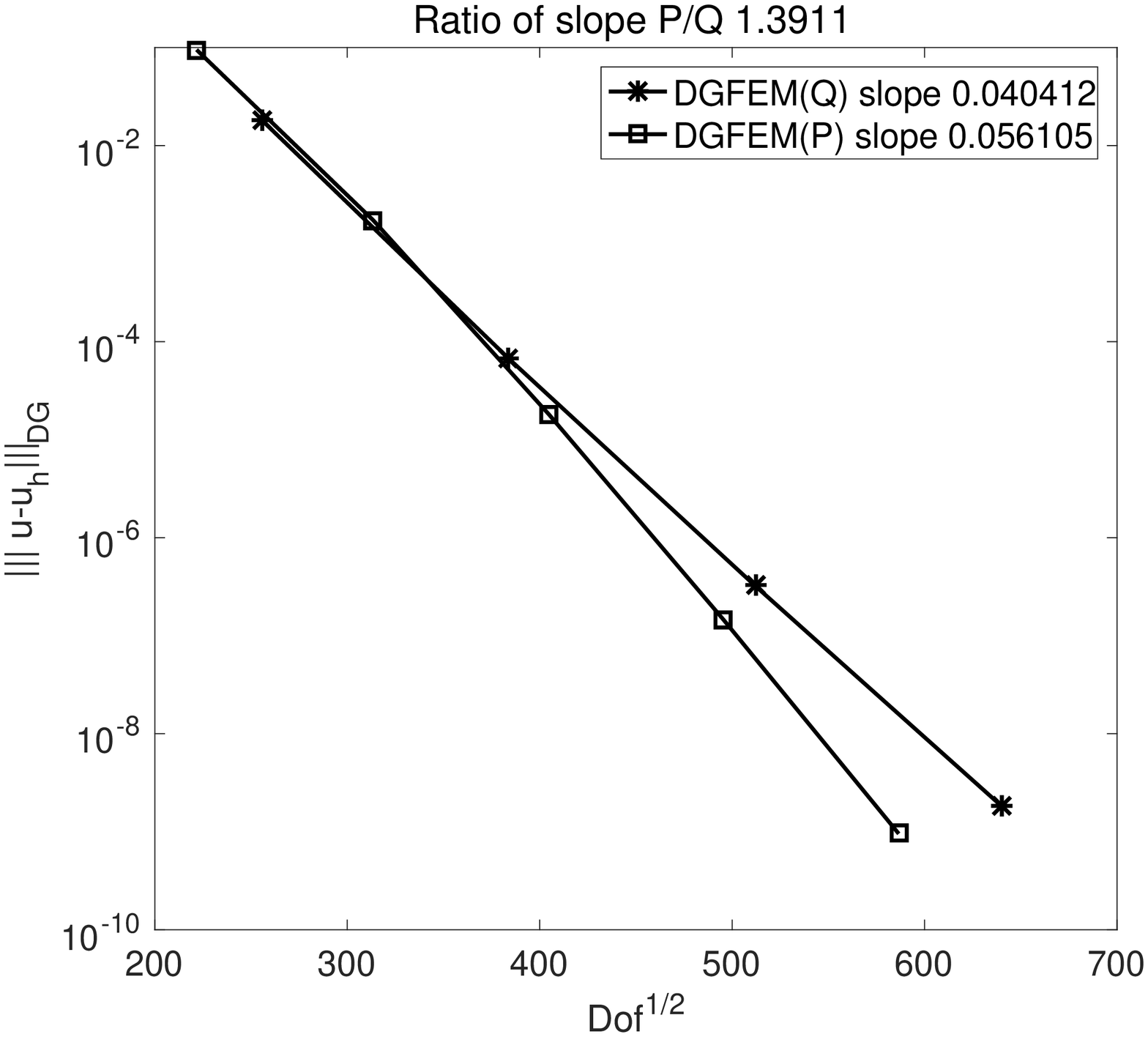}
\caption{Example 1: Convergence of the DGFEMs   under $p$-refinement on uniform square elements   {($\ncdg{u - u_h}$).
  $8\times8$ mesh (left); $64\times64$ mesh (right); $128\times128$ mesh (bottom). }}\label{ch7:Convection-diffusion} 
\end{figure}

By following \cite{cangiani2015hp}, we use the symmetric interior penalty DGFEMs    employing  a special class of quadrilateral meshes for which the discontinuity in the exact solution lies on  element interfaces. In this setting,   we modify the discontinuity-penalization parameter $\sigma$,  
so that $\sigma$ vanishes on edges which form part of the interface $y=0$; this ensures that the (physical) discontinuity present in the exact solution is not penalized within by the numerical scheme.  

In this case, the exact solution is piecewise analytic on the  two parts of the domain. In Figure \ref{ch7:Convection-diffusion},  we  observe that  { the DG--norm $\ncdg{u - u_h}$  decays exponentially  for both DGFEM(P)  and DGFEM(Q) under $p$-refinement on $64$, $4096$ and $16384$ uniform square elements. The definition of DG--norm $\ncdg{\cdot}$  can be found in \cite{cangiani2015hp}. 
}Moreover,  the slope of the convergence line  for the DGFEM(P)  is greater than the line of  DGFEM(Q)  in error  against $\sqrt{Dof}$. The ratio between the two slopes is about $1.39$ on coarse meshes and fine meshes.   {The numerical observation confirms the theoretical results in Theorem \ref{better slope}. }

\subsection{Example 2}

In the second example, we investigate the computational efficiency of  FEM(S) and FEM(Q) on standard tensor-product elements
(quadrilaterals in 2D and hexahedra in 3D).

Firstly, we consider the following two--dimensional Poisson  problem: let $\Omega=(0,1)^2$ and select $f=2 \pi^2 \sin(\pi x) \sin(\pi y)$, so that the exact solution
 is given by $u = \sin(\pi x) \sin(\pi y)$.

In this case, the exact solution is piecewise analytic on the domain. In Figure \ref{ch7:Elliptic 2D FEM},  we  observe that  the $H^1$--seminorm  {$|u-u_h|_{H^1(\Omega)}$ decays exponentially  for both FEM(S)  and FEM(Q) under $p$-refinement on  $64$, $4096$ and $16384$ uniform square elements.} Again, we observe that the slope of the convergence line for the FEM(S)  is greater than the line of  FEM(Q)  in error  against $\sqrt{Dof}$. The ratio between the two slopes is about $1.39$ on coarse meshes and fine meshes.

%

\begin{figure}[t]
\centering
\includegraphics[width=0.45\linewidth]{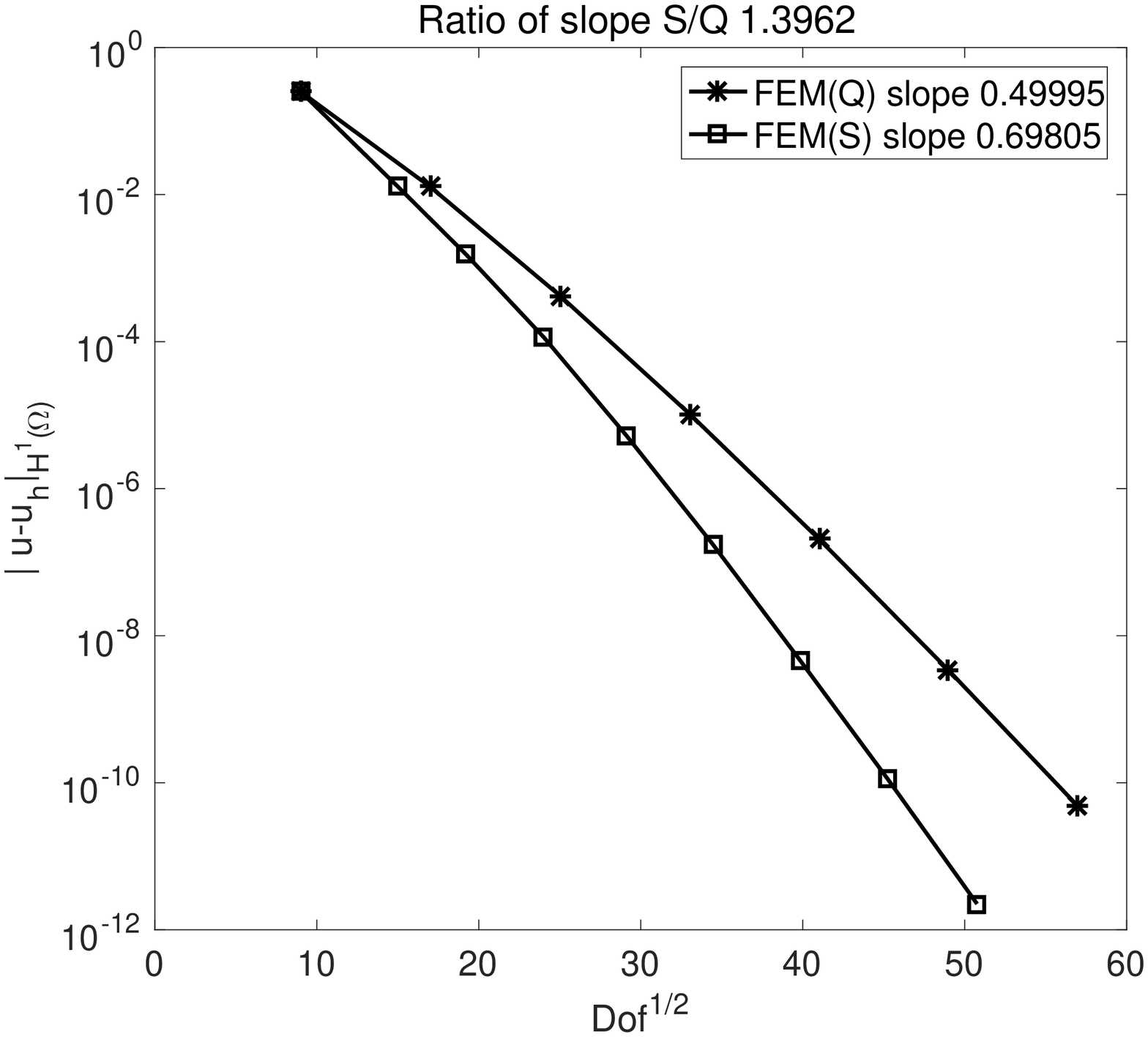}
\hspace{0cm}
\includegraphics[width=0.45\linewidth]{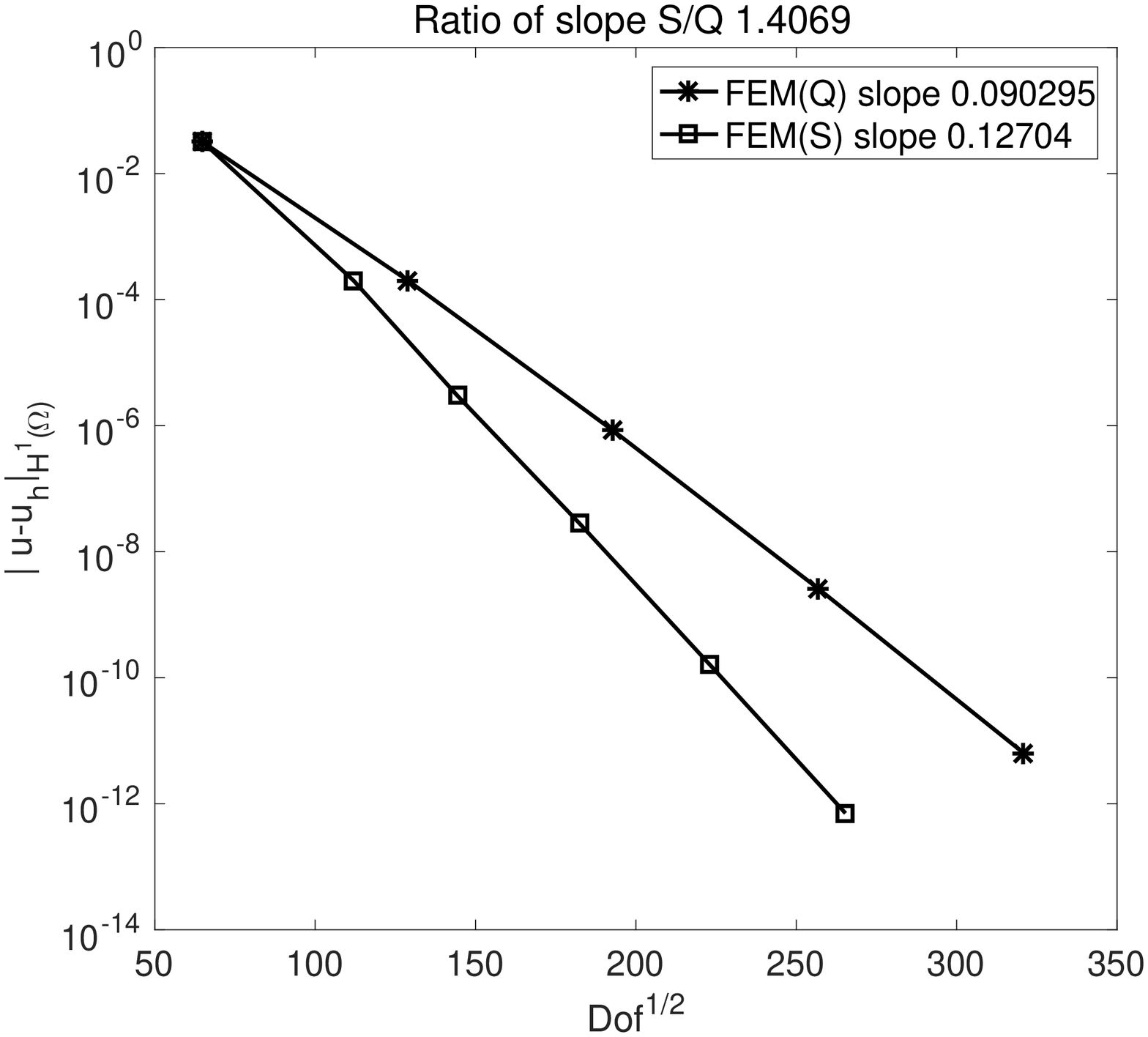} 
\\
\includegraphics[width=0.45\linewidth]{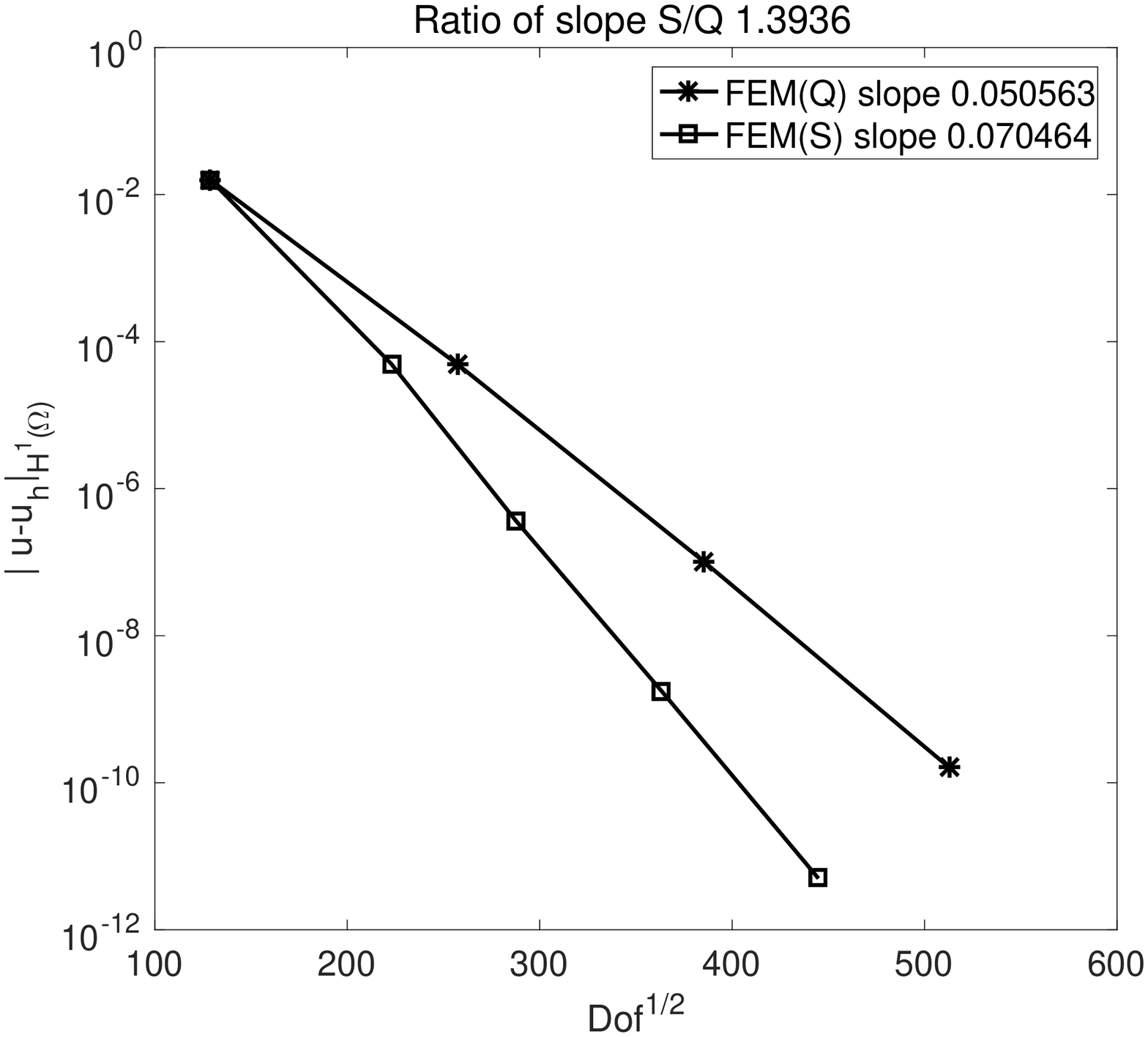} 
\caption{Example 2: Convergence of the FEMs under $p$-refinement on uniform square  elements. 
  {($|u-u_h|_{H^1(\Omega)}$).
  $8\times8$ mesh (left); $64\times64$ mesh (right); $128\times128$ mesh (bottom). }}\label{ch7:Elliptic 2D FEM} 
\end{figure}

%

\begin{figure}[t]
\centering
\includegraphics[width=0.45\linewidth]{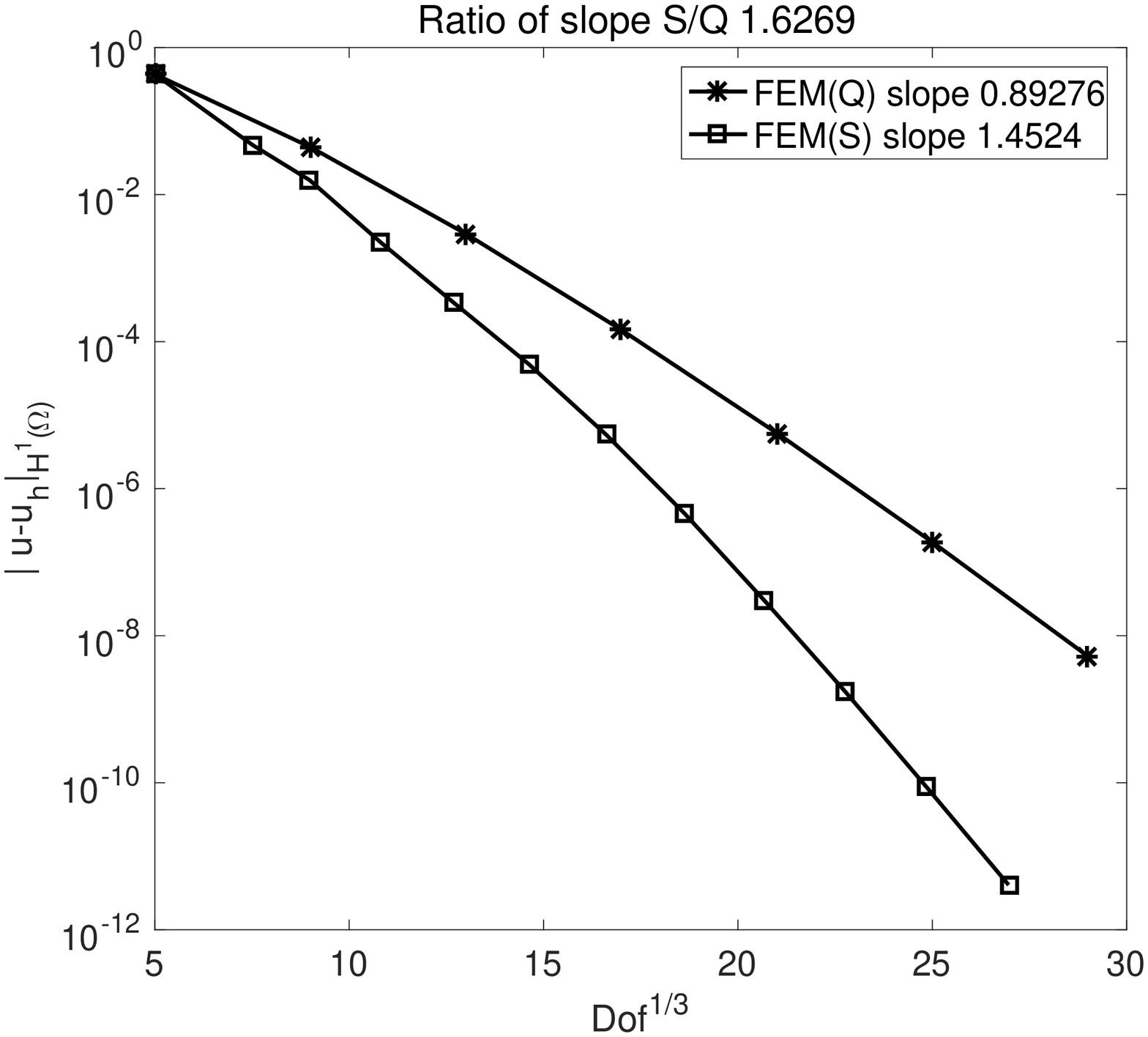} 
\hspace{0cm}
\includegraphics[width=0.45\linewidth]{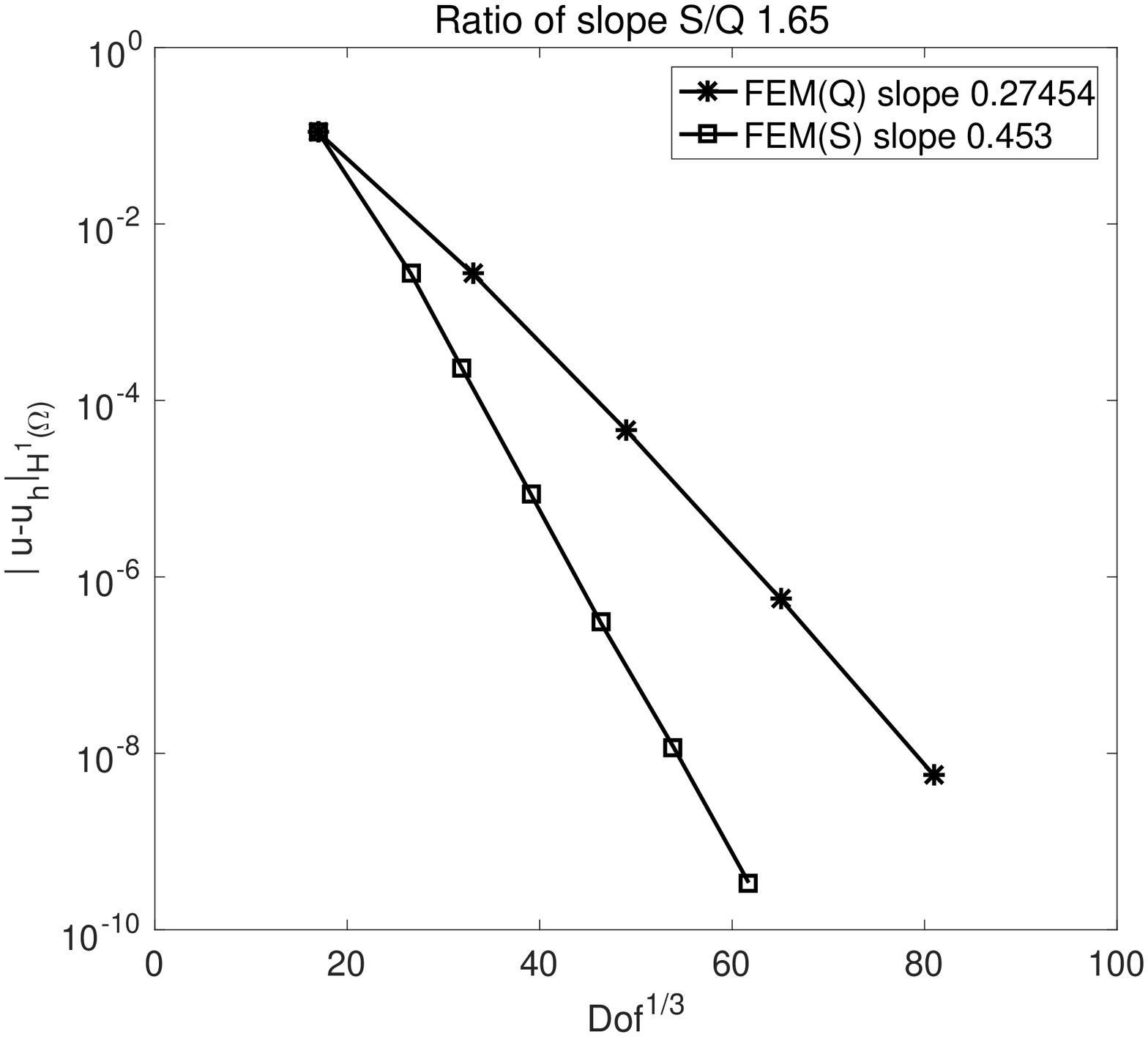}
\\
\includegraphics[width=0.45\linewidth]{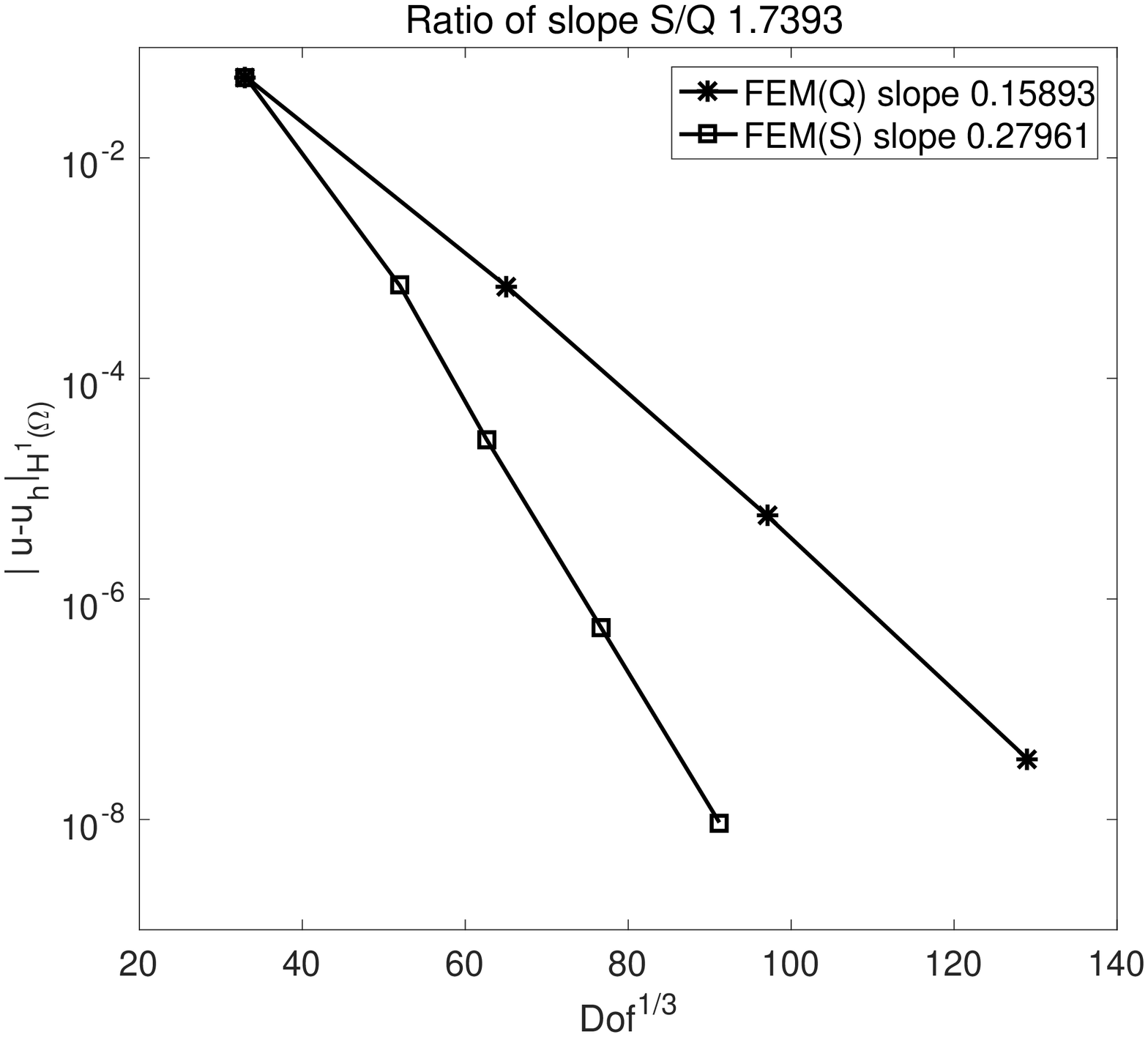}
\caption{Example 2: Convergence of the FEMs under $p$-refinement on uniform hexahedral elements.  {($|u-u_h|_{H^1(\Omega)}$). $4\times4\times4$ mesh (left); $16\times16\times16$ mesh (right);  $32\times32\times32$ mesh (bottom).} }\label{ch7:Elliptic 3D FEM}
\end{figure}

We now consider the three--dimensional variant of the above problem. Let $\Omega=(0,1)^3$ and select $f=3 \pi^2 \sin(\pi x) \sin(\pi y) \sin(\pi z)$, so that the exact solution is given by $u = \sin(\pi x) \sin(\pi y) \sin(\pi z)$. 

 In Figure \ref{ch7:Elliptic 3D FEM},  we  observe that  the $H^1$--seminorm  {$|u-u_h|_{H^1(\Omega)}$ decays exponentially  for both FEM(S)  and FEM(Q) under $p$-refinement on $64$, $4096$ and $32768$ uniform hexahedral elements.} Moreover, we observe that the slope of the convergence line for the FEM(S)  is greater than the line of  FEM(Q)  in error  against $\sqrt[3]{Dof}$. The ratio between the two slopes is about $1.62$ on coarse meshes and  $1.73$ on fine meshes.  {The numerical observation confirms the theoretical results in Theorem \ref{better slope}.}

\subsection{Example 3}

In the third example, we investigate the convergence behaviour of the  FEM(S) and FEM(Q) approaches for the Poisson problem on a non-smooth domain with  fixed computational meshes under $p$-refinement. To this end, we let $\Omega$ be the L-shaped domain
$(-1,1)^2\setminus [0,1)\times (-1,0]$.   Uniform square meshes consisting of $12$ elements are used. 
Then, writing $(r,\varphi)$ to denote the system of polar coordinates, we
impose an appropriate inhomogeneous boundary condition for $u$
so that
$$
u = r^{2/3} \sin(2\varphi/3);
$$
cf. \cite{Wihler2002}. We note that $u$ is analytic in
$\overline{\Omega}\setminus\{{\bf 0}\}$, but $\nabla u$ is
singular at the origin; indeed, here $u \not\in H^2(\Omega)$.
This example reflects the
typical (singular) behaviour that solutions of elliptic boundary value
problems exhibit in the vicinity of reentrant corners in the computational
domain.

In fact, $u \in H^{\frac{5}{3}-\epsilon}(\Omega)$, for any $\epsilon>0$.  We investigate  the convergence rate of the FEM(S) and FEM(Q) under $p$-refinement for this problem.  In Table \ref{p-rate}, we list the  $H^1$--seminorm error and also the convergence rate of FEM(S) and FEM(Q) with polynomial order $p=1,\dots,60$. We point out that due to the singularity at the origin,  geometrically graded quadrature points towards  the origin are used in order to get the desired accuracy (see \cite{chernov2011exponential}).

 \begin{table}[!htb] 
\begin{center}
\begin{small}
\begin{tabular}{|c|c|c|c|c|c|} 
 \hline  
$p$  & \multicolumn{2}{|c|}{FEM(S) }   &  \multicolumn{2}{|c|}{FEM(Q) } & Ratio of Error \\
 \hline   
   & $|{u-u_h}|_{H^1(\Omega)}$   &  $p$-\text{rate} &$|{u-u_h}|_{H^1(\Omega)}$&$p$-\text{rate}& FEM(S)/FEM(Q) \\
 \hline   
1	&	2.09E-01	&		&	2.09E-01	&		&	1	\\
\hline
2	&	1.25E-01	&	0.7386	&	9.62E-02	&	1.1204	&	1.303	\\
\hline
3	&	1.20E-01	&	0.1096	&	5.99E-02	&	1.1691	&	2.0023	\\
\hline
4	&	9.00E-02	&	0.9971	&	4.23E-02	&	1.2087	&	2.128	\\
\hline
5	&	6.93E-02	&	1.1703	&	3.21E-02	&	1.2372	&	2.16	\\
\hline
10	&	2.96E-02	&	1.261	&	1.32E-02	&	1.2968	&	2.2311	\\
\hline
15	&	1.76E-02	&	1.2921	&	7.79E-03	&	1.3143	&	2.2558	\\
\hline
20	&	1.21E-02	&	1.306	&	5.33E-03	&	1.3215	&	2.2675	\\
\hline
25	&	9.03E-03	&	1.3135	&	3.97E-03	&	1.3251	&	2.2741	\\
\hline
30	&	7.10E-03	&	1.3181	&	3.12E-03	&	1.3272	&	2.2783	\\
\hline
35	&	5.79E-03	&	1.3211	&	2.54E-03	&	1.3285	&	2.2811	\\
\hline
40	&	4.86E-03	&	1.3232	&	2.13E-03	&	1.3294	&	2.2832	\\
\hline
45	&	4.15E-03	&	1.3247	&	1.82E-03	&	1.33	&	2.2847	\\
\hline
50	&	3.61E-03	&	1.3259	&	1.58E-03	&	1.3305	&	2.2859	\\
\hline
55	&	3.18E-03	&	1.3268	&	1.39E-03	&	1.3309	&	2.2868	\\
\hline
60    &      2.84E-03	&	1.3275	&	1.24E-03	&	1.3312	&	2.2876	\\
\hline   
\end{tabular}  
\end{small}
 \end{center}
\caption{Example 3: Convergence rate in $p$ of the FEM(S) and FEM(Q) in $H^1$--seminorm.} \label{p-rate}
 \end{table}

As we can see,  the convergence rate in $p$ for both FEM(S) and FEM(Q) are approximately $\mathcal{O}(p^{4/3})$. The convergence rate in $p$ is double  the theoretical rate  with respect to the Sobolev regularity of $u$. This is the \emph{doubling order} convergence in the $p$-version finite element, see \cite{Babuska-Suri:SINUM:1987,schwab} for details. The reason of this doubling order convergence in $p$ is related  to the fact that standard Sobolev space can not optimally characterize the singularity of $r^\gamma \log^\nu r$ type, $\gamma \in \mathbb{R}^+$, $\nu\in \mathbb{N}$; indeed from \cite{babuska2002direct_part1,babuska2002direct_part2,MR1754720}, we know that the modified Jacobi-weighted Besov spaces provide a sharper  function space setting  to characterize such  singular functions.  By using the results in \cite{babuska2002direct_part2}, FEM(Q)  has the following sharp error bound under $p$-refinement for this problem 
$$
c p^{-\frac{4}{3}}\leq |u-u_h|_{H^1(\Omega)}\leq C p^{-\frac{4}{3}}, 
$$
where  constants $C$ and $c$ are independent of $p$.  Moreover, we observe that  the FEM(S)  error is greater than FEM(Q)  error by a factor about $2.29$ for fixed $p$. By noting that $2^{4/3}\approx 2.52$, we  suppose that  the error bound for FEM(S) satisfying the following relation  
$$
|u-u_h|_{H^1(\Omega)}\leq \tilde{C} \Big( \frac{2}{p}\Big)^{\frac{4}{3}}, 
$$
The proof  the above intuitive optimal $p$-version error bound for FEM(S) is beyond the scope of this work, which depends on the  approximation theory for orthogonal projections onto the ${\cal P}_p$ basis  and the ${\cal S}_p$ basis in the  modified Jacobi-weighted Besov spaces. 

\begin{figure}[t]
\begin{center}
\begin{tabular}{cc}
\hspace{-0.8cm}
\includegraphics[width=0.5\linewidth]{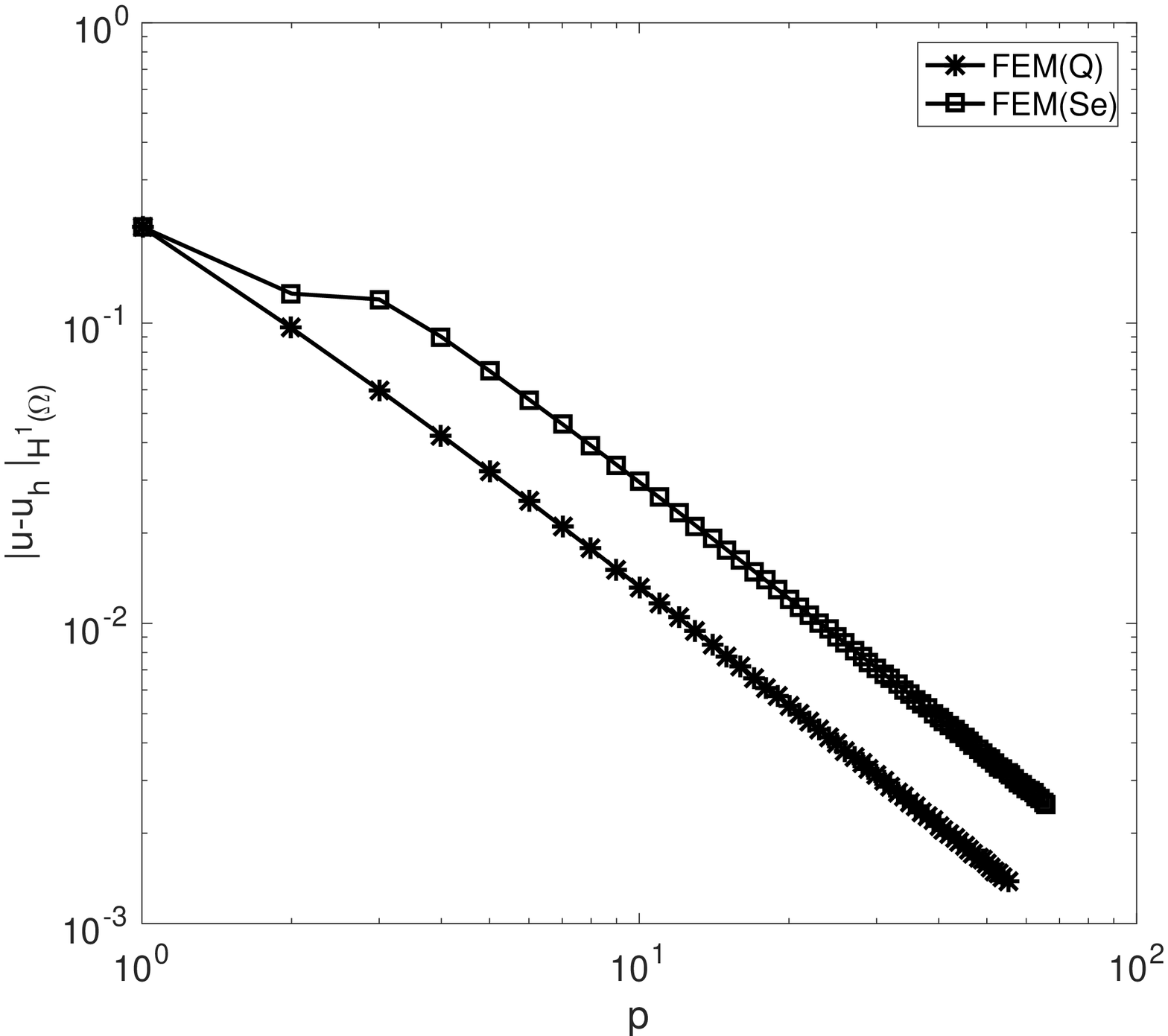} &
\hspace{-0.3cm}
\includegraphics[width=0.5\linewidth]{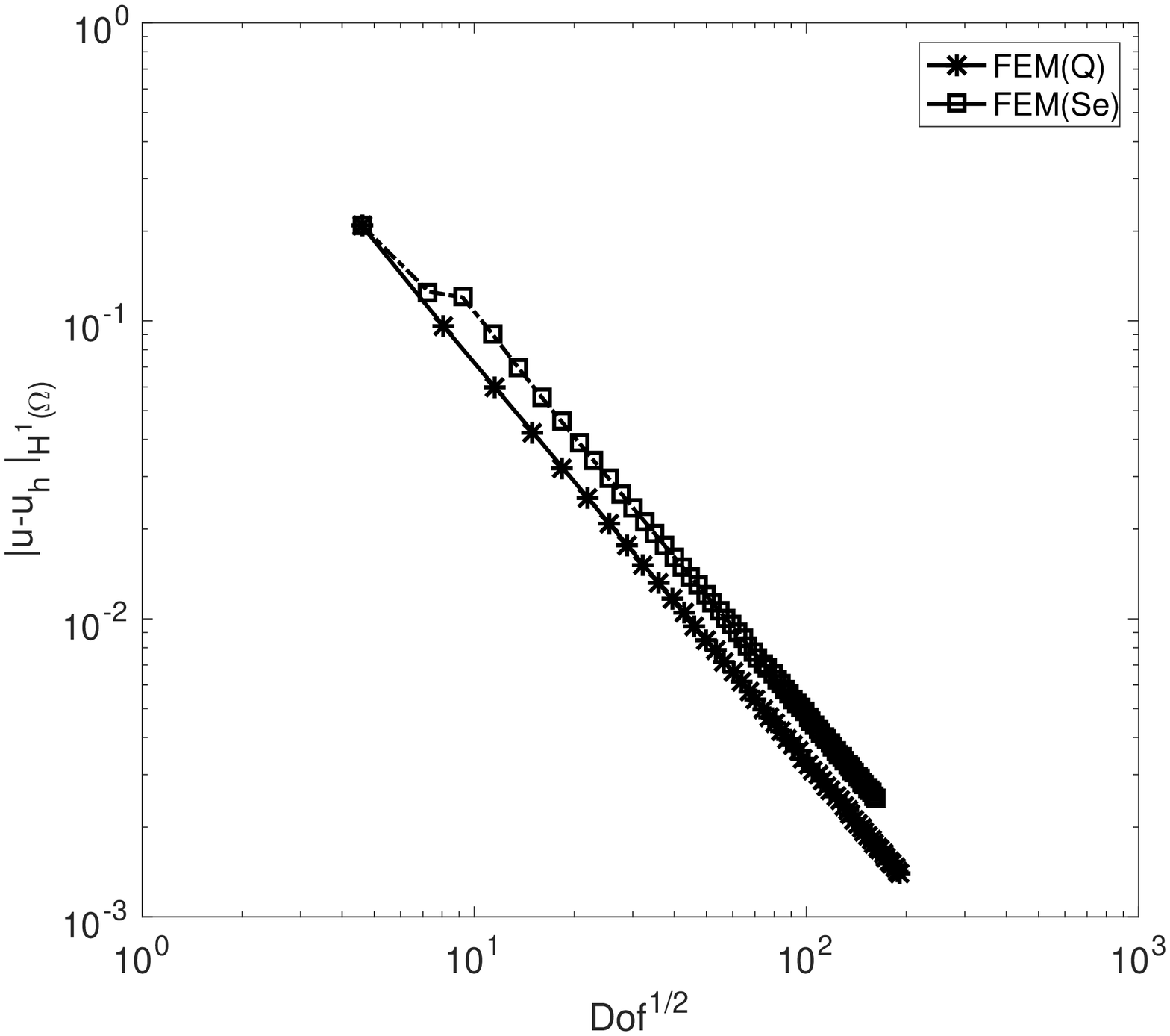} \\
\end{tabular}
\end{center}
\caption{Example 3: Convergence of the FEMs under $p$-refinement on $12$ uniform square elements. Error against $p$ (left);  Error against ${Dof}^{1/2}$ (right).}\label{ch7:Elliptic Lshape} 
\end{figure}

Now, let's consider the relation of the error against $\sqrt{Dof}$. By using the asymptotic relation \eqref{dof Q}, we have the following relation for FEM(Q)
$$
 |u-u_h|_{H^1(\Omega)}\leq C Dof^{-\frac{2}{3}}.
$$
Then, we derive the following relation for FEM(S) by  using the  relation \eqref{dof S},
$$
 |u-u_h|_{H^1(\Omega)}\leq \tilde{C} \Big(\frac{2}{\sqrt{2Dof}}\Big)^{\frac{4}{3}} 
 = \tilde{C}  \Big(\frac{2}{{Dof}}\Big)^{\frac{2}{3}}.
$$
Hence, both FEM(S) and FEM(Q) have the  same algebraic convergence rate in degrees of freedom under $p$-refinement. However, the error bound of FEM(S) seems to be  larger than the error bound of FEM(Q) by a constant.  This result is observed in the Figure \ref{ch7:Elliptic Lshape}, where we can see that the convergence line of FEM(S) and FEM(Q) have the same slope in error against $p$ (left) and error against ${Dof}^{1/2}$ (right). But the error of FEM(S) is always larger than the error of FEM(Q) in both graphs.  

\section{Conclusions and remarks}
In this work, we derived new $hp$-optimal approximation results for the $L^2$-orthogonal projection onto the total degree ${\cal P}_p$ basis, and the $H^1$-projection onto the serendipity basis ${\cal S}_p$ over tensor product elements.  With these results, we proved that  the exponent of the exponential rate of convergence with respect to the  number of degrees of freedom  for the  DGFEM employing  ${\cal P}_p$ basis and  the serendipity FEM  are greater than the exponent  of the exponential rate of convergence for their counterparts employing ${\cal Q}_p$ basis  for analytic functions under $p$-refinement. Moreover, the exponent for  the  ${\cal P}_p$ and ${\cal S}_p$ bases are   larger than that of ${\cal Q}_p$ basis by a constant only depending on dimension. The sharpness of the theoretical results has been verified by numerical examples. 

Finally, we remark on some applications and potential extensions of results in this   work. First, we note that the $hp$-optimal error bounds for the $L^2$-orthogonal projections onto the  ${\cal P}_p$ basis and $H^1$-projections onto the  ${\cal S}_p$ basis can be used to improve the $hp$-error bounds for  mixed-FEMs employing the ${\bf BDFM}$-elements and virtual element methods, see \cite{MR1974174} and \cite{MR3509090}. Next, as we have already observed in the numerical example $7.2$ in \cite{cangiani2013hp},  the $hp$-adaptive DGFEM employing the ${\cal P}_p$ basis gives a greater exponent of the  exponential rate of convergence in terms of number of degrees of freedom than the exponent of  DGFEM employing the  ${\cal Q}_p$ basis for solving the  L-shaped domain problem. We are very interested in  applying the results in this work for  $hp$-version FEMs employing the serendipity basis and DGFEMs employing the ${\cal P}_p$ basis  in a general $hp$-refinement setting. 
 

\begin{acknowledgement}
 The author wishes to express his gratitude to Emmanuil Georgoulis (University of Leicester \& National Technical University of Athens) and Andrea Cangiani (University of Leicester) for their helpful comments.  {Z. D. was supported by the Leverhulme Trust (grant no. RPG-2015-306).} 
\end{acknowledgement}

\bibliographystyle{plain} 
\bibliography{HPFEM}


\section*{Appendix. Technical results}
We present the details for  proving  that $T_{2,a} = \pi_p^{(1)}\pi_p^{(2)}u -  \pi_{{\cal S}_p}^{(1,2)}u$ in relation \eqref{trace term 3D}. The idea is that both of the $H^1$-projections $ \pi_p^{(1)}\pi_p^{(2)}u$  and  $\pi_{{\cal S}_p}^{(1,2)}u$ are independent of the variable $x_3$ for $u\in H^3(\hat{\k})$, $\hat{\k}=(-1,1)^3$. By recalling the definition of the $2$D $H^1$-projection $ \pi_p^{(1)}\pi_p^{(2)}u$ in \eqref{def:H^1 projector Q basis}, then we have the following relation   
\begin{align} \label{2D H1 Q for 3D function}
\pi_p^{(1)}\pi_p^{(2)}u (x_1,x_2,x_3)
&:= \sum_{i_1=0}^{p-1}\sum_{i_2=0}^{p-1} \tilde{a}_{i_1 i_2}(x_3)
\psi_{i_1} (x_1)\psi_{i_2}(x_2)  
+ \sum_{{i_1}=0}^{p-1} \tilde{b}_{{i_1}}(x_3)\psi_{i_1}(x_1) 
\nno \\
&\quad + \sum_{{i_2}=0}^{p-1} \tilde{c}_{{i_2}}(x_3)\psi_{i_2}(x_2)
 +u(-1,-1,x_3) .
\end{align}
with $\tilde{a}_{{i_1}{i_2}}(x_3)$, $\tilde{b}_{i_1}(x_3)$ and  $\tilde{c}_{i_2}(x_3)$ given by: 
{
\begin{align} \label{coefficients 2D in 3D}
\tilde{a}_{{i_1}{i_2}}(x_3)&= \frac{2{i_1}+1}{2}\frac{2{i_2}+1}{2}
\int_{-1}^1 \int_{-1}^1    \partial_1  \partial_2 u(x_1,x_2,x_3) L_{i_1}(x_1) L_{i_2}(x_2)\ud x_1 \ud x_2,  \nno \\
\tilde{b}_{i_1} (x_3)&= \frac{2{i_1}+1}{2} \int_{-1}^1    \partial_1  u(x_1,-1,x_3) L_{i_1}(x_1)\ud x_1, \nno  \\
\tilde{c}_{i_2} (x_3)&= \frac{2{i_2}+1}{2} \int_{-1}^1    \partial_2  u(-1,x_2,x_3) L_{i_2}(x_2)\ud x_2 .
\end{align}
The above coefficients in \eqref{coefficients 2D in 3D} all depend the variable $x_3$. We note that for function $u\in H^3(\hat{\k})$, the above expansion can be rewritten in following way, 
\begin{align} \label{2D H1 Q for 3D function part 2}
&\pi_p^{(1)}\pi_p^{(2)}u  (x_1,x_2,x_3) \nno \\
&= \sum_{i_1=0}^{p-1}\sum_{i_2=0}^{p-1} 
\Big(\int_{-1}^{x_3} \partial_3 \tilde{a}_{i_1 i_2}(x_3) \ud x_3 \Big)
\psi_{i_1} (x_1)\psi_{i_2}(x_2)   
+ \sum_{i_1=0}^{p-1}\sum_{i_2=0}^{p-1} 
 \tilde{a}_{i_1 i_2}(-1)
\psi_{i_1} (x_1)\psi_{i_2}(x_2)  
\nno \\
&  
+ \sum_{{i_1}=0}^{p-1} 
\Big(\int_{-1}^{x_3} \partial_3 \tilde{b}_{{i_1}}(x_3) \ud x_3\Big) \psi_{i_1}(x_1) 
+ 
\sum_{{i_1}=0}^{p-1} \tilde{b}_{{i_1}}(-1)\psi_{i_1}(x_1) 
\nno \\
& 
+\sum_{{i_2}=0}^{p-1} 
\Big(\int_{-1}^{x_3} \partial_3 \tilde{c}_{{i_2}}(x_3) \ud x_3\Big) \psi_{i_2}(x_2) 
+ 
\sum_{{i_2}=0}^{p-1} \tilde{c}_{{i_2}}(-1)\psi_{i_2}(x_2) \nno \\
& +  \int_{-1}^{x_3}  \partial_3 u(-1,-1,x_3) \ud x_3
+   u(-1,-1,-1) .
\end{align}
Next,  we expand above partial derivative of coefficients $\tilde{a}_{{i_1}{i_2}}(x_3)$, $\tilde{b}_{i_1}(x_3)$ and  $\tilde{c}_{i_2}(x_3)$ into Legendre  polynomials  $L_{i_3}(x_3)$,  and compare   coefficients in  \eqref{coefficients 2D in 3D}, \eqref{coefficients 3D} and \eqref{coefficients 3D part 2}, the following relation holds
\begin{align} \label{def:H^1 projector Q basis 2D}
& \pi_p^{(1)}\pi_p^{(2)}u (x_1,x_2,x_3) \nno \\
& = \sum_{i_1=0}^{p-1}\sum_{i_2=0}^{p-1} \sum_{i_3=0}^{\infty} a_{i_1 i_2i_3}\psi_{i_1} (x_1)\psi_{i_2}(x_2)\psi_{i_3}(x_3)  
 + \sum_{{i_1}=0}^{p-1} \sum_{{i_2}=0}^{p-1} b_{i_1i_2}\psi_{i_1}(x_1) \psi_{i_2}(x_2) \nno \\
&  +\sum_{{i_1}=0}^{p-1} \sum_{{i_3}=0}^{\infty} c_{i_1i_3}\psi_{i_1}(x_1) \psi_{i_3}(x_3)  
    +\sum_{{i_2}=0}^{p-1} \sum_{{i_3}=0}^{\infty} d_{i_2i_3}\psi_{i_2}(x_2) \psi_{i_3}(x_3) 
  + \sum_{i_1=0}^{p-1} e_{i_1} \psi_{i_1}(x_1) 
   \nno \\
    &  
    +  \sum_{i_2=0}^{p-1} f_{i_2} \psi_{i_2}(x_2)
  +  \sum_{i_3=0}^\infty g_{i_3} \psi_{i_3}(x_3)+u(-1,-1,-1).  
\end{align}
Similarly, by using the same techniques, in conjunction with  the definition of the  $2$D $H^1$-projection $\pi_{{\cal S}_p}^{(1,2)}u$ in  \eqref{def:H^1 projector S basis}.  The following relation holds for $\pi_{{\cal S}_p}^{(1,2)}u$
\begin{align} \label{def:H^1 projector S basis 2D}
&\pi_{{\cal S}_p}^{(1,2)}u (x_1,x_2,x_3) \nno \\
& =  \sum_{\substack {  {i_1}+{i_2} = 2 \\{i_1}\geq1, {i_2}\geq1 }} ^{p-2 }
\Big( \sum_{i_3=0}^{\infty} a_{i_1 i_2i_3}\psi_{i_1} (x_1)\psi_{i_2}(x_2)\psi_{i_3}(x_3) +b_{i_1i_2}\psi_{i_1}(x_1) \psi_{i_2}(x_2)  \Big)
 \nno  \\
 &+  \sum_{i_1=0}^{p-1} \Big(  \sum_{i_3=0}^\infty  a_{i_1 0i_3}\psi_{i_1} (x_1)\psi_{0}(x_2)\psi_{i_3}(x_3) +b_{i_10}\psi_{i_1}(x_1) \psi_{0}(x_2) \Big) \nno \\
  &   +  \sum_{i_2=1}^{p-1} \Big(  \sum_{i_3=0}^\infty  a_{0i_2i_3}\psi_{0} (x_1)\psi_{i_2}(x_2)\psi_{i_3}(x_3) +b_{0 i_2}  \psi_{0}(x_1) \psi_{i_2}(x_2)  \Big) \nno \\
&  +\sum_{{i_1}=0}^{p-1} \sum_{{i_3}=0}^{\infty} c_{i_1i_3}\psi_{i_1}(x_1) \psi_{i_3}(x_3) 
  +\sum_{{i_2}=0}^{p-1} \sum_{{i_3}=0}^{\infty} d_{i_2i_3}\psi_{i_2}(x_2) \psi_{i_3}(x_3)  
 \nno \\
 &
 + \sum_{i_1=0}^{p-1} e_{i_1} \psi_{i_1}(x_1)
 +  \sum_{i_2=0}^{p-1} f_{i_2} \psi_{i_2}(x_2)
  +  \sum_{i_3=0}^\infty g_{i_3} \psi_{i_3}(x_3)+u(-1,-1,-1). 
\end{align}
By combining above two relation \eqref{def:H^1 projector Q basis 2D} and \eqref{def:H^1 projector S basis 2D}, we derive the desired relation,
\begin{align} \label{Result for T2a}
&(\pi_p^{(1)}\pi_p^{(2)}u -\pi_{{\cal S}_p}^{(1,2)}u) (x_1,x_2,x_3) \nno \\
&= \sum_{\substack {{i_1}+{i_2} = p-1  \\p-1 \geq {i_1}\geq1, p-1 \geq {i_2}\geq1 }} ^{2(p-1)}
\Big( \sum_{i_3=0}^{\infty} a_{i_1 i_2i_3}\psi_{i_1} (x_1)\psi_{i_2}(x_2)\psi_{i_3}(x_3)  +b_{i_1i_2}\psi_{i_1}(x_1) \psi_{i_2}(x_2)  \Big) \nno \\
&= T_{2,a}.
\end{align}
The proof is complete.

\end{document}